\documentclass[12pt,a4paper,reqno]{amsart}
\usepackage[utf8]{inputenc}
\usepackage[a4paper,margin=1in
]{geometry}
\usepackage[english]{babel}
\usepackage{float}

\usepackage[normalem]{ulem}
\usepackage{cancel}

\usepackage{verbatim}

\usepackage[colorinlistoftodos, textwidth=3cm]{todonotes}
\newcounter{dccomment}

\usepackage{latexsym}    
\usepackage{amssymb}   
\usepackage{amsmath}    
\usepackage{amsbsy}
\usepackage{amsthm}
\usepackage{amsgen}
\usepackage{amsfonts}
\usepackage{bbm}
\usepackage{array}
\usepackage{lmodern}

\usepackage{graphicx}
\usepackage{caption}
\usepackage[labelformat=simple]{subcaption}

\newsavebox{\largestimage}

\usepackage{tikz}
\usepackage{tikz-cd}
\usetikzlibrary{decorations.pathmorphing}
\usetikzlibrary{calc, decorations.markings}

\usepackage{xcolor}

\usepackage{verbatim}

\usepackage[shortlabels]{enumitem}

\usepackage{lineno}

\usepackage[nostamp]
{draftwatermark}
\SetWatermarkText{In Preparation}
\SetWatermarkScale{0.5}
\SetWatermarkColor[gray]{0.9}

\newcommand{\1}{\mathbbm{1}} 
\newcommand{\II}{\mathrm{II}} 
\newcommand{\ar}{\mathrm{ar}} 
\newcommand{\C}{\mathbb{C}} 
\newcommand{\Ch}{\mathrm{Ch}} 
\newcommand{\diam}{\mathrm{diam}} 
\newcommand{\D}{\mathbb{D}} 
\newcommand{\F}{\mathsf{F}} 	
\newcommand{\fol}{\mathcal{F}} 
\newcommand{\g}{\mathtt{g}} 
\newcommand{\G}{\mathcal{G}} 
\newcommand{\hh}{\mathtt{h}} 
\newcommand{\Hol}{\mathrm{Hol}} 
\newcommand{\hor}{\mathcal{H}} 
\newcommand{\Id}{\mathrm{Id}} 
\newcommand{\Iso}{\mathrm{Isom}} 
\newcommand{\kernel}{\mathrm{Ker}} 
\newcommand{\m}{\mathfrak{m}} 
\newcommand{\N}{\mathbb{N}} 
\newcommand{\Ns}{\mathsf{N}} 
\newcommand{\ob}{\mathrm{ob}} 
\newcommand{\Orth}{\mathrm{O}} 
\newcommand{\Pp}{\mathsf{P}} 
\newcommand{\prin}{\mathrm{prin}} 
\newcommand{\proj}{\mathrm{proj}} 
\newcommand{\QQ}{\mathtt{Q}} 
\newcommand{\R}{\mathbb{R}} 
\newcommand{\reg}{\mathrm{reg}} 
\newcommand{\Sec}{\mathrm{Sec}} 
\newcommand{\SO}{\mathrm{SO}} 
\newcommand{\Sh}{\mathrm{Sh}} 
\newcommand{\Sp}{\mathbf{S}} 
\newcommand{\ver}{\mathcal{V}} 
\newcommand{\Vol}{\mathrm{Vol}} 
\newcommand{\Z}{\mathbb{Z}} 

\newtheorem{theorem}{Theorem}  
\newtheorem{corollary}[theorem]{Corollary}  

	\newtheorem{thm}{Theorem}[section]
	
	\newtheorem{lemma}[thm]{Lemma}
	\newtheorem{cor}[thm]{Corollary}
	
	\newtheorem{proposition}[thm]{Proposition}
	
	\newtheorem{problem}[thm]{Problem}
	\newtheorem{conjecture}[thm]{Conjecture}

\theoremstyle{definition}	
	\newtheorem{remark}[thm]{Remark}
	\newtheorem{definition}[thm]{Definition}
	\newtheorem{example}[thm]{Example}
	
 \newtheoremstyle{TheoremNum}
        {\topsep}{\topsep}              
        {\itshape}                      
        {}                              
        {\bfseries}                     
        {.}                             
        { }                             
        {\thmname{#1}\thmnote{ \bfseries #3}}
    \theoremstyle{TheoremNum}
    \newtheorem{duplicate}{Theorem}
    \newtheorem{duplicateCOR}{Corollary}

\usepackage[hyphens]{url}    
\usepackage{hyperref}
\hypersetup{
    colorlinks=true,
    linkcolor=blue,
    citecolor=blue,
    urlcolor=blue,
    pdfauthor={Diego Corro},
    pdftitle={Singular Riemannian foliations and collapse}
}
\usepackage[hyphenbreaks]{breakurl}

\usepackage[reftex]{theoremref}
\title[Singular Riemannian foliations and collapse]{Singular Riemannian foliations and collapse
}


\author[D.~Corro]{Diego Corro$^{\ast}$}
\address[D. CORRO]{School of Mathematics, Cardiff University, United Kingdom}
\urladdr{\url{www.diegocorro.com}}
\email{\href{mailto:diego.corro.math@gmail.com}{diego.corro.math@gmail.com}}
\thanks{$^\ast$Supported  DFG-Eigenestelle Fellowship CO 2359/1-1, and a UKRI Future Leaders
Fellowship [grant number MR/W01176X/1; PI: J Harvey].}

\begin{document}

\overfullrule = 5pt

\begin{abstract}
In this survey we present classical results on methods to use group actions to collapse manifolds to the orbit spaces while keeping some control on the curvature, and recent extensions of these constructions to the setting of singular Riemannian foliations.
\end{abstract}

\maketitle


\section{Introduction}


\emph{Regular foliations} are generalizations of fiber bundles and by the Frobenius Theorem, a regular foliation corresponds to the maximal level sets for solutions to a non-singular linear system of PDE's (the PDE's determine an integrable distribution in the tangent bundle). In \cite{Reinhart1959} Reinhart introduced a particular type of regular foliation, the so-called (regular) Riemannian foliations, where the leaves are locally equidistant with respect to a Riemannian metric. Riemannian submersions are examples of Riemannian foliations \cite{GromollWalschap}. 

When considering a partition of the a manifold by injectively immersed submanifolds whose tangent spaces depend continuously on the base point and the dimension of the tangent spaces of the leaves is non-constant, we obtain a so-called \emph{singular foliation}, studied by Sussmann \cite{Sussmann73} and Stefan \cite{Stefan1974}. These foliations correspond to integrable singular linear systems of PDE's. When we add the geometric hypothesis of the leaves being locally equidistant we obtain a so-called \emph{singular Riemannian foliation} (see Section \ref{Singular Riemannian foliations} for a precise definition).

It was proven that given a regular Riemannian foliation, the foliation induced by taking the closure of its leaves is a singular Riemannian foliation \cite{Molino}. Thus singular Riemannian foliations are  generalizations of regular Riemannian foliations. Another large family of examples of singular Riemannian foliations is given by Lie group actions by isometries \cite{AlexandrinoBettiol}. 

In general, singular Riemannian foliations show remarkable analogies with Lie group actions by isometries,  as pointed out in \cite[p. 185]{Molino}. For example, a singular Riemannian foliation induces on the manifold a stratification by the dimension of the leaves, and when the leaves are closed the quotient space of the foliation inherits a nice geometric structure (see \cite{Corro}). But as shown in \cite{FerusKarcherMuenzner1981,Radeschi2014}, there are infinitely many singular Riemannian foliations that are not induced by group actions by isometries nor by Riemannian submersions. Thus singular Riemannian foliations are strictly more general than group actions by isometries.

Since the 1980's there has been a continuous study of singular Riemannian foliations, in particular topological properties of the foliations and also how the presence of a singular Riemannian foliation constrains the topology of the foliated manifold, see for example \cite{Alexandrino2010,AlexandrinoBriquetToeben2012,AlexandrinoInagakiStruchiner2018,GalazGarciaRadeschi2015,Corro2019,CorroMoreno2020,GeRadeschi2013,Lytchak2010,RadeschiThesis}. More recently there has been interest in the interplay between singular Riemannian foliations and the existence of foliated solutions to PDE's \cite{CorroFernandezPerales22,AlexandrinoCavenaghiCorroInagaki2024}. Thus, we can consider singular Riemannian foliations as a generalized notion of symmetry for Riemannian manifolds. Moreover, by \cite{CorroGalazGarcia2024} this generalized notion of symmetries is compatible with the classical notion of group actions by isometries.  

For the study of manifolds with positive (non-negative) curvature, Grove has proposed the so-called \emph{symmetry program} \cite{Grove2002}: We should try to study the structure of manifolds with positive (non-negative) curvature which have a large degree of symmetry in some sense (see for example \cite{GroveSearle1994,GroveZiller2002,FangRong2005,Rong2002,Wilking2003}). Part of this program consists on using the presence of symmetries to construct metrics with positive curvature, for example \cite{CorroGalazGarcia2016,GilkeyParkTuschmann1998,GroveZiller2000,GroveZiller2002,SearleSolorzanoWilhelm2015}. 

In several of the aforementioned results techniques for deforming geometries along the orbits of a group action  while preserving control of the curvature are exploited. To to the best of the author knowledge, there are few references in the literature on the subject of deforming foliations, see for example \cite{CheegerGromov1990,FarrellJones1998,CrainicMestreStruchiner2020,delHoyoFernandes2018,EpsteinRosenberg1977,Hamilton1978,GroveKarcher1973}.

In this survey we present two recent deformation results \cite{Corro2024,Corro2025} by the author for singular Riemannian foliations. In \cite{Corro2024} closed singular foliations with flat leaves are considered. This are part of the so-called \emph{$A$-foliations}, closed singular Riemannian foliations with aspherical leaves, introduced in \cite{GalazGarciaRadeschi2015}, and further studied in \cite{Corro2019}, as generalizations of torus actions by isometries. In the simply-connected case by \cite{Corro2019}, the fundamental group of the leaves are \emph{Bieberbach groups}; i.e.  discrete co-compact torsion free subgroups of the
isometry group of Euclidean spaces, and correspond to the fundamental groups of flat manifolds. Thus it is natural to consider the extra hypothesis of actually having flat geometries on the leaves, but there are examples of closed singular Riemannian foliations with aspherical leaves that do not admit such a flat metric structure \cite{FarrellWu2018}.

The main result in \cite{Corro2024} is that a closed regular singular foliation with flat leaves on a simply-connected manifold is given by an (almost-free) smooth effective torus action. A key result to achieve this conclusion is that a Riemannian manifold $(M,\mathtt{g})$ equipped with closed regular foliation $\mathcal{F}$ with flat leaves can be collapsed with bounded curvature and bounded diameter by shrinking the leaves of the foliation. Recall that a Riemannian manifold $(M,\mathtt{g})$ can be \emph{collapsed with uniformly bounded curvature} if there exists a sequence $\{\mathtt{g}_i\}_{i\in \N}$ of Riemannian metrics on $M$ such that $\g_1=\g$, $\mathrm{inj}(\g_n)\to$ as $i\to\infty$, and there exists $\Lambda\geq 0$ such that $\|\Sec(\g_i)\|\leq \Lambda$. When there also exists $D>0$ such that $\mathrm{diam}(M,\g_i)\leq D$, we say that $(M,\g)$ \emph{collapses with bounded diameter}.
\begin{theorem}[Theorem B in \cite{Corro2024}]\th\label{T: Collapse of regular foliation}
Consider $(M, \fol, \g)$ a regular closed non-trivial foliation such that for $L\in \fol$ we have $(L,\g|_{L})$ is a flat manifold. For $0<\delta\leq 1$ and $p\in M$ given we consider $T_p M= T_pL_p \oplus \nu_p (L_p)$ and set 
\[
\g_\delta(p) = \delta^2 (\g(p)|_{T_p L_p})\oplus (\g(p)|_{\nu_p(L_p)}).
\]
Then $(M,\g_\delta)$ collapses with uniformly bounded curvature and uniformly bounded diameter.
\end{theorem}

In this survey we also present a slight generalization of this result to closed singular Riemannian foliations with flat leaves, in the case when the singular strata are isolated from each other, and the leaves have positive dimension.

\begin{theorem}\th\label{MT: Collapsing general flat foliations}
Consider $(M, \fol, \g)$ a  closed singular Riemannian foliation with flat leaves on a compact Riemannian manifold, such that all singular strata are isolated from one another. Then we can collapse $(M,\g)$  with bounded curvature, in a way that also the volume of the collapsing sequence goes to $0$.
\end{theorem}

The lack of a uniform diameter bound is a big constrain, since it does not allow us to follow the same approach as in \cite{Corro2024} to obtain rigidity conclusions about the deformation, namely \cite[Claim 3.2.2]{Corro2024} fails. Nonetheless, we still obtain that the collapsing metrics can be approximated by Riemannian metrics which are invariant with respect to an $\F$-structure when $M$ has finite fundamental group \cite{CheegerFukayaGromov1992} (see Section~\ref{S: Classical deformations} for definitions).

\begin{corollary}\th\label{MC: collapse controlled by torus action}
Consider $(M, \fol, \g)$ a  closed singular Riemannian foliation with flat leaves on a compact Riemannian manifold with finite fundamental group, such that all singular strata are isolated from one another, and assume that $\lambda\leq \Sec(\g)\leq \Lambda$.  Let $\{\g_i\}_{i\in \N}$ be the collapsing sequence given by \th\ref{MT: Collapsing general flat foliations}. Then, given $\varepsilon>0$ there exists $N:=N(\varepsilon)\in \N$ such that for all $i\geq N$ there exists an $F$-structure $\fol_{\varepsilon,i}$ on $M$ and an $\fol_{\varepsilon,i}$-invariant Riemannian metric $\g_{\varepsilon,i}$ such that
\[
e^{-\varepsilon}\g_i<\g_{\varepsilon,i}< e^{\varepsilon}\g_i,
\]
and
\[
\lambda-\varepsilon\leq \Sec(\g_{\varepsilon,i})\leq \Lambda+\varepsilon.
\]
\end{corollary}

\begin{remark}
If we remove the hypothesis of $M$ being simply-connected in \th\ref{MC: collapse controlled by torus action}, we recover similar conclusions, but instead of an $\F$-structure we get an $\Ns$-structure, which is a more general structure.
\end{remark}

Another general deformation procedure for actions by isometries by compact Lie groups was presented by Cheeger in \cite{Cheeger1973} (see also \cite{Mueter}), inspired by Berger's example of collapsing the unit round sphere $\Sp^3$ to a $2$-sphere $\Sp^2$ of radius $1/2$. The basic idea is to ``shrink'' the length of the vectors tangent to the orbits, while keeping the length of the vectors normal to the orbits unchanged and either preserving or increasing (depending on the geometry of the compact Lie group acting) the lower bound on the sectional curvature of the original manifold (see  \cite{MoullieWebpage} for a nice visual explanation). Due to this last property, this deformation technique has been used to show the existence of metrics with positive lower sectional or Ricci curvature (for example \cite{CavenaghiESilvaSperanca2023,LawsonYau1972,SearleSolorzanoWilhelm2015,Searle2023}).

An algebraic generalization of Lie groups and smooth Lie group actions are Lie grou\-poids, and Lie grou\-poid actions. Groupoids were introduced and named by Brant in \cite{Brandt1927}, and Lie groupoids were introduced as a tool in differential topology and geometry by Ehresmann \cite{EhresmannCompltes} by adding differentiable structures. Moreover there are strong connections between singular Riemannian foliations and Lie groupoid actions \cite{Debord2001,Moerdijk}. Namely for a regular Riemannian foliation there exists the so-called \emph{holonomy groupoid}, whose orbits are the leaves of the foliation \cite{Moerdijk}. For  a closed singular Riemannian foliation $(M,\fol,\g)$, it was proven in \cite{AlexandrinoInagakiStruchiner2018}, that for a fixed leaf $L\in \fol$  the so-called ``linearized foliation'' defined in a tubular neighborhood of $L$ is given by the orbits of a Lie groupoid action. The linearized foliation is determined by the holonomy of the leaf, and the biggest possible homogeneous foliation contained in the infinitesimal foliation at any point in the leaf. In \cite[p. 210]{Molino} orbit-like foliations were introduced. These are singular Riemannian foliations, for which the linearized foliations of all leaves agree with the foliation \cite{AlexandrinoInagakiStruchiner2018}. A large family of examples of orbit-like foliations are closed singular Riemannian foliations of codimension $1$.

In \cite{Corro2025} the author introduced a generalization of Cheeger deformations for proper Lie groupoid actions, such that under a technical hypothesis we can compute the evolution of the sectional curvature in an analogous fashion to the classical Cheeger deformations for compact Lie group actions by isometries.

\begin{theorem}\th\label{MT: Cheeger deformation collapses}
Consider $\G\rightrightarrows M$ a proper Lie groupoid acting on $\alpha\colon P\to M$ a submersion, and $\eta^{(1)}$ a Riemannian metric on $\G$ making it a Riemannian Lie groupoid. Let $\eta^P$ a left-$\G$-invariant Riemannian metric on $P$, and assume that for all $p\in M$ we have $\nu_p(L_p)\subset T_p \alpha^{-1}(\alpha(p))$. Then for the Cheeger deformation $(P,\eta_\varepsilon)$ we have for $v,w\in T_pP$ linearly independent vectors, with $v = X^\ast(p)+v^\perp$ and $w = Y^\ast(p)+w^\perp$ for some $x,y\in \kernel(D_{\1_{\alpha(p)}}s)$. Then
\begin{linenomath}
\begin{align}\label{EQ: Full curvature description}
\begin{split}
K_{\eta_\varepsilon}\Big(\Ch^{-1}_\varepsilon(p)&(v),\Ch^{-1}_\varepsilon(p)(w)\Big)  = K_{\eta^P}(v,w)+\varepsilon^3K_{\eta^{(1)}}\Big(\Sh(p)(v),\Sh(p)(w)\Big)\\
+&3\Big\|A_{h(p)(v)}h(p)(w)\Big\|^2_{\left(\frac{1}{\varepsilon}\eta^{(1)}+\eta^P\right)}\\
+&\Big\|\II_\varepsilon\big((v,-\varepsilon \Sh(p)(v)),(w,-\varepsilon\Sh(p)(w))\big)\Big\|^2_{\frac{1}{\varepsilon}\eta^{(1)}+\eta^P}\\
 -&\left(\frac{1}{\varepsilon}\eta^{(1)}+\eta^P\right)\Big(\II_\varepsilon\big((v,-\varepsilon\Sh(p)(v)),(v,-\varepsilon\Sh(p)(v))\big),\\
 &\II_\varepsilon\big((w,-\varepsilon\Sh(p)(w)),(w,-\varepsilon\Sh(p)(w))\big)\Big).
\end{split}
\end{align}
\end{linenomath}
where $\II_\varepsilon(\cdot,\cdot)$ is the second fundamental form of $(\G\times_M P,\widehat{\eta}_\varepsilon)\subset (\G\times P,(1/\varepsilon)\eta^{(1)}+\eta^{P})$, and $A$ is the $A$-tensor of the Riemannian submersion $\bar{t}\colon (\G\times_M P,\hat{\eta}_\varepsilon)\to (P,\eta_t)$.
\end{theorem}

As mentioned above, around a fixed leaf of a closed singular Riemannian foliation there exists a subfoliation given by a Lie groupoid representation. In the case when the foliation is an orbit-like foliation, i.e. this subfoliation agrees with the foliation, then it is natural to consider the following problem.

\begin{problem}
Let $(M,\fol,\g)$ be a closed orbit-like singular Riemannian foliation. Can we construct a global deformation procedure which agrees with the local Cheeger deformations presented here, while keeping global control over the evolution of the sectional curvature?
\end{problem}

We point out that due to the existence of foliated partitions of unity (e.g. \cite{CorroFernandezPerales22}), we can always glue these deformations to obtain a global deformation, but losing control on the sectional curvature.

The previous problem is relevant in the case of closed singular Riemannian foliations of codimension $1$ on compact Riemannian manifolds. These manifolds have the  topology of a gluing of the disjoint union of two disk bundles with diffeomorphic boundaries, via a diffeomorphism of their boundary. It has been conjectured that manifolds with positive sectional curvature admit a closed singular Riemannian foliations of codimension $1$.

\begin{conjecture}[\cite{GonzalezAlvaroGuijarro2023,Grove2002}]\th\label{ConjectureGrove}
Let $(M,\g)$ be a simply connected Riemannian manifold with positive (non-negative) sectional curvature. Then there exists a codimension one singular Riemannian foliation  $\fol$  (possibly with respect to a different Riemannian metric)on $M$.
\end{conjecture}

We present the double disk decomposition of a codimension one closed singular Riemannian foliation on a compact manifold in Section~\ref{S: Connection between Singular Riemannian foliations and Lie groupoid actions}, and what is a necessary condition get a global Cheeger deformation for these type of singular Riemannian foliations. 


This manuscript is organized as follows. In Section~\ref{S: Classical deformations} we present the Gromov-Hausdorff topology, as well as two classical notions of geometric deformations: Cheeger deformations for compact Lie group actions by isometries, and collapsing with bounded curvature as well as $\F$-structures. In Section~\ref{Singular Riemannian foliations} we introduce singular Riemannian foliations. In Section~\ref{S: Collapse of singular Riemannian foliations with flat leaves} we present proofs of \th\ref{T: Collapse of regular foliation} and \th\ref{MT: Collapsing general flat foliations}. Then in Section~\ref{S: Lie groupoid actions} we present the concepts of Lie group actions and their geometric properties needed to present in Section~\ref{S: Cheeger like deformation for Lie groupoids actions} generalized Cheeger deformations, and the results connected to \th\ref{MT: Cheeger deformation collapses}.


\section*{Acknowledgements}
The author thanks Marcos M. Alexandrino, Fernando Galaz-Garc\'{i}a, Karsten Grove, John Harvey, Alexander Lytchak, Jes\'{u}s N\'{u}\~{n}ez-Zimbr\'{o}n, Jaime Santos, Ilohann Speran\c{c}a, Wilderich Tuschmann and Marco Zambon for useful conversations.


\section{Classical deformations}\label{S: Classical deformations}

In this section we present classical deformations for torus and compact Lie group actions by isometries.

\subsection{Gromov-Hausdorff distance and collapse}

We begin by considering first $(X,d_X)$ and $(Y,d_Y)$ metric spaces. We define the \emph{Gromov-Hausdorff distance} between $(X,d_X)$ and $(Y,d_Y)$, denoted by 
\[
d_{\mathrm{GH}}((X,d_X),(Y,d_Y)),
\] 
to be the infimum  of all Hausdorff distances $d_H(f(X),g(Y))$ for all metric spaces $(Z,d_Z)$  and all isometric embeddings $f\colon X\to Z$ and $g\colon Y\to Z$ (see  \cite[Section 7.3]{BuragoBuragoIvanov}). Recall that
\begin{linenomath}
\begin{align*}
d_{\mathrm{H}}(f(X),g(Y)):= \max\Big\{&\sup\big\{\inf\{d_Z\big(f(x),g(y)\big)\mid y\in Y\}\,\big|\, x\in X\big\},\\
&\sup\big\{\inf\{d_Z\big(f(x),g(y)\big)\mid x\in Y\}\,\big|\, y\in X\big\}\Big\}.
\end{align*}
\end{linenomath}

For compact metric spaces we have the following characterization of the Gromov-Hausdorff distance: given $\varepsilon >0$ we say that a subset $S\subset X$ is an \emph{$\varepsilon$-net} if we have that $d_X(x,S) = \inf\{d_X(x,s)\mid s\in S\}\leq \varepsilon$. Given $X$, $Y$ two compact metric spaces, and $\varepsilon,\delta>0$, we say that \emph{$X$ and $Y$ are $(\varepsilon,\delta)$-approximations of each other} if there exists $\{x_i\}_{i=1}^N\subset X$, $\{y_i\}_{i=1}^N\subset Y$ such that 
\begin{enumerate}
\item The sets $\{x_i\}_{i=1}^N$, $\{y_i\}_{i=1}^N$ are $\varepsilon$-nets,
\item $|d_X(x_i,x_j)-d_Y(y_i,y_j)|<\delta$ for all $i,j\in \{1,\ldots,N\}$.
\end{enumerate}
In the case when $\varepsilon= \delta$, we say that \emph{$X$ and $Y$ are $\varepsilon$-approximations of each other}. With this we can write the characterization of convergence in the Gromov-Hausdorff sense for compact spaces.
\begin{proposition}[Proposition 7.4.11 in \cite{BuragoBuragoIvanov}]\th\label{P: characterization of GH convergence via approximations}
Let $X$ and $Y$ be compact metric spaces. Then we have
\begin{enumerate}[(1)]
\item If $Y$ is an $(\varepsilon,\delta)$ approximation of $X$, then $d_{\mathrm{GH}}(X,Y)< 2\varepsilon+\delta$.
\item If $d_{\mathrm{GH}}(X,Y)<\varepsilon$, then $Y$ is a $5\varepsilon$-approximation of $X$.
\end{enumerate}
\end{proposition} 

Given a smooth compact manifold $M$, we say that a sequence of Riemannian metrics $\{\g_n\}_{n\in \N}$ on $M$ \emph{collapses} if for any $p\in M$ the injectivity radius of $\g_n$ at $p$ goes to $0$ as $n\to \infty$. 

It is not difficult to give examples of collapsing: We can consider any compact Riemannian manifold $(M,\g)$ with diameter $D$ and set $\g_n:=\frac{1}{n}\g$. Then we have that $\diam(\g_n)= \frac{1}{\sqrt{n}}D$. Since the diameter of $\g_n$ goes to $0$ as $n\to \infty$, then the injectivity radius also goes to $0$. In this particular case the sequence $\{(M,\g_n)\}_{n\in\N}$ is converging in the Gromov-Hausdorff sense to the metric space $\{\ast\}$ consisting of only a point. Moreover, we have that the sectional curvature of $\g_n$ is $\Sec(\g_n)=n\Sec(\g)$; thus unless $(M,g)$ is  a flat manifold, we have that $|\Sec(\g_n)|$ goes to $+\infty$. This means that in some sense the geometry of $\g$ is being ``blown up''.

Thus it is of interest to give procedures or characterizations of collapse while keeping some control over the curvature of the manifold. We say that a Riemannian manifold $(M,\g)$ \emph{collapses with bounded curvature} if there exists a collapsing sequence of Riemannian metrics $\{\g_n\}_{n\in\N}$ with $\g_1=\g$ and there exists $\Lambda\geq 0$ such that  $|\Sec(\g_n)| \leq \Lambda$ for all $n\in\N$.

\subsection{Riemannian submersions and Cheeger deformations}\label{S: Cheeger deformations}

We start by presenting the classical theory of Cheeger deformations, and of Riemannian submersions.

Let $M$ be a  smooth manifold, and $G$ a compact Lie group. A smooth map $\mu\colon H\times M\to M$ is called a \emph{smooth action of $G$ on $M$} if the following hold:
\begin{enumerate}[(i)]
    \item For $e\in G$ the identity element, and any $p\in M$ we have $\mu(e,p) = p$. 
    \item For any $g_1,g_2\in G$ and any $p\in M$ we have $\mu(g_1,\mu(g_2,p)) = \mu(g_1 g_2, p)$.
\end{enumerate}
Given a fixed point $p \in M$, the \emph{orbit through $p$} is the set $G(p) = \{\mu(g,p)\mid h\in G\}$. The \emph{isotropy subgroup} at $p$ is the subgroup of elements of $G$ that fix $p$, i.e.  $G_p = \{g\in G\mid \mu(g,p) = p\}$. When  the map $\tilde{\mu}\colon G\times M\to M\times M$ given by $\tilde{\mu}(g,p) = (p,\mu(g,p))$ is a proper map, i.e.  the preimage of a compact set under $\tilde{\mu}$ is compact, we say that the action $\mu$ is a \emph{proper action}. For a proper action  the isotropy groups $G_p$ are closed subsets of $G$. Moreover, for a proper action we have that any orbit $G(p)$ is diffeomorphic to $G/G_p$.

Given a Riemannian metric $\g$ on $M$, we say that the (smooth) action $\mu\colon G\times M\to M$ is by isometries if for any $h\in G$ we have that the diffeomorphism $\mu_h\colon M\to M$ given by $\mu_h(p) = \mu(h,p)$ is an isometry of $(M,\g)$, i.e. $(\mu_h)^\ast(\g) = \g$. In particular the partition of $M$ by the orbits of $G$, i.e. $\fol_G = \{G(p)\mid p\in M\}$, is such that the orbits are locally equidistant. We refer to the quotient space $M/G$ consisting of all orbits of the action of $G$, equipped with the quotient topology as the \emph{orbit space of the action}. 

Given a smooth action $\mu\colon G\times M\to M$, we can define a smooth action $\overline{\mu}\colon G\times (G\times M)\to G\times M$  as follows:
\[
    \overline{\mu}(g_1,(g_2,p)) = (g_1 g_2,\mu(g_1,p)).
\]

We point out that the action $\overline{\mu}$ is free: fix $(g,p)\in G\times M$ and assume that there exists $\tilde{g}\in G$ such that  $\overline{\mu}(\tilde{g},(g,p)) = (g,p)$. Then we have that $\tilde{g}g = g$. This implies that $\tilde{g} = e$. Moreover, it can be proven that the map $\tilde{\phi}\colon (G\times M)\to M$ defined as
\[
	\tilde{\phi}(g,p) = \mu(g^{-1},p)
\]
is a submersion. For $p\in M$ fixed, we have the following description of the fiber $\tilde{\phi}^{-1}(p) = \{(g,\mu(g,p)\mid g\in G\} = G(e,p)$. 

Recall that given a smooth submersion $f\colon M\to N$ and a Riemannian metric $\g$ on $M$, $f$ is called an \emph{Riemannian submersion} if the Lie derivative of $\g$ in the directions tangent to the fibers of $f$ is zero. That is, the fibers are locally equidistant. Recall that given a Riemannian submersion $f\colon (M,\g)\to N$ we obtain a new Riemannian metric $\hh$ on $N$ \cite{GromollWalschap}.

Assume now that $G$ acts by isometries on $(M,\g)$, and that $G$ admits a bi-invariant metric $\QQ$. We consider $(G\times M, \frac{1}{t} \QQ+\g)$ and observe that for each $t>0$ the metric  $\frac{1}{t}\QQ+\g$ is $G$-invariant. Observe that for $(g,p)\in G\times M$ we have that $G(g,p) = G(e,g^{-1}p) = G(e,p')$. Thus, we conclude that for any $t>0$, the submersion $\tilde{\phi}\colon (G\times M, \frac{1}{t}\QQ+\g)\to M$ is a Riemannian submersion. 

Thus we obtain a family $\{\g_t\}_{t>0}$ of Riemannian metrics on $M$, such that 
\[
\tilde{\phi}\colon \left(G\times M,\frac{1}{t}\QQ+\g\right)\to (M,\g_t)
\]
is a Riemannian submersion. We call $\g_t$ the \emph{Cheeger deformation metric} of $\g$.

\begin{remark}
Observe that the construction of a Cheeger deformation depends on the choice of the bi-invariant metric $\QQ$. But the general ``geometric properties'' of $\QQ$ do not depend on the choice made, since the moduli space of bi-invariant metric on a compact Lie group is contractible \cite{FloresTorres2024}.
\end{remark}

Given a vector $x\in T_e G \cong \mathfrak{g}$ we define the \emph{action vector field $X^\ast(p)\in T_p M$ of $x$ at $p\in M$} as
\[
X^\ast(p):=D_{(e,p)}\mu(X,0).
\] 
With this, for each $p\in M$ we define a symmetric endomorphism $\Sh(p)\colon T_e G\to T_e G$, called the \emph{shape tensor} (or \emph{orbit tensor}) as the operator that satisfies the following  identity:
\[
	\QQ(\Sh(p)(x),y) = \g(X^\ast(p),Y^\ast(p))\quad \mbox{for all } y\in T_e G.
\]
Given an action of $G$ by isometries on $(M,\g)$, for $p\in M$ we denote by $\nu_p (G(p))$ the normal bundle to the orbits. That is, all tangent directions of points in the orbit which are $\g$-perpendicular to the directions tangent to the orbit. 

Given $t>0$ and $p\in M$, we define the \emph{Cheeger tensor} $\Ch_t(p)\colon T_pM\to T_pM$ as follows: Given any vector $v\in T_p M$, there exists $x\in T_e G$ and $\xi\in \nu_p(G(p))$ such that $v= X^\ast(p)+\xi$. Then we set
\[
\Ch_t(p)(v) =\Big((\Id+t\Sh(p))^{-1}(x)\Big)^\ast(p)+\xi.
\]
This is an invertible tensor, and moreover by \cite[Satz 3.3]{Mueter} we have that for any $t\geq 0$ and $v,w\in T_p M$ the following holds:
\[
\g_t(v,w) = \g(\Ch_t(p)(v),w).
\]

Recall that given a Riemannian submersion $f\colon (M,\g)\to (N,\hh)$ we can define two distributions on $M$. The the \emph{vertical distribution}  $\ver = \kernel\, D f$ and the \emph{horizontal distribution} $\hor = \kernel\, D\pi^\perp \subset TM$, i.e. the $\g$-perpendicular distribution to $\ver$. Observe that $TM = \ver\oplus \hor$. Given a vector field $X\in \mathfrak{X}(M)$, we denote  the \emph{vertical component} by  $X^v\in \ver$, and  the \emph{horizontal component} by $X^h\in \hor$.

Given a Riemannian submersion $\pi\colon (M,\g)\to (N,\hh)$ we have the following relationships between the sectional curvature of the metrics $\g$ and $\hh$.

\begin{thm}[O'Neill's formula \cite{ONeill1966}, see Section 1.5 in \cite{GromollWalschap}]\th\label{T: ONeills formula}
Let $f\colon (M,\g)\to (N,\hh)$ be a Riemannian submersion. Then for any linearly independent $X,Y\in \mathfrak{X}(N)$ it holds
\begin{linenomath}
\begin{align*}
K(\hh)(X,Y) =& K(\g)(\widetilde{X},\widetilde{Y})+3\Big\|A_{\widetilde{X}}\widetilde{Y}\Big\|^2_{\g}\\
&=K(\g)(\widetilde{X},\widetilde{Y})+\frac{3}{4}\Big\|[\widetilde{X},\widetilde{Y}]^v\Big\|^2_{\g}.
\end{align*}
\end{linenomath}
Here $\widetilde{X},\widetilde{Y}\in \hor$ are the unique horizontal lifts of $X,Y$, and 
\[
K(\hh)(X,Y) := \hh(R_{\hh}(X,Y)Y,X),
\]
where $R_{\hh}$ is the Riemannian curvature tensor of $\hh$ (analogous for $\g$).
\end{thm}

Applying this result, we obtain the following expression for the sectional curvature of $\g_t$.

\begin{thm}[Section~3.d in \cite{Mueter}]\th\label{T: Classical cheeger deformation classical}
Consider $(M,\g)$ a Riemannian, and $G$ a Lie group acting smoothly and effectively by isometries on $M$. We consider $\QQ$ a bi-invariant metric on $G$. Then for the Cheeger deformation metric $\g_t$, we have for any $v=X^\ast(p)+\xi$, $w=Y^\ast(p)+\zeta\in T_pM$ the following
\begin{linenomath}
\begin{align}\label{EQ: Cheeger deformation curvature}
\begin{split}
K(\g_t)\Big(\Ch_t^{-1}(p)(v),\Ch^{-1}_t(p)(w)\Big) =& t^3 K(\QQ)(\Sh(p)(x),\Sh(p)(y))+ K(\g)(v,w)\\
&+3\|A_{(t\Sh(p)(x),v)}(t\Sh(p)(y),w)\|^2_{\frac{1}{t}\QQ+\g}\\
=& \frac{t^3}{4}\Big\|[\Sh(p)(x),\Sh(p)(y)]\Big\|^2_\QQ+K(\g)(v,w)+\\ 
&+\frac{3}{4}\Big\|\big[(t\Sh(p)(x),v),(t\Sh(p)(y),w)\big]^v\Big\|^2_{\frac{1}{t}\QQ+\g}.
\end{split}
\end{align}
\end{linenomath}
\end{thm}

\subsection{Collapse with bounded curvature \texorpdfstring{$\F$}{F}-structures}

Recall that a compact Riemannian manifold $(M,\g)$ \emph{collapses with bounded curvature}, if there is a collapsing sequence $\{\g_n\}_{n\in \N}$ and there exists $0\leq \Lambda\in \R$ such for all $n\in \N$ we have $|\Sec(\g_n)|\leq \Lambda$.

This notion of collapse was studied in \cite{CheegerGromov1986,CheegerGromov1990,CheegerFukayaGromov1992}. In \cite{CheegerGromov1986} $\F$-structures were introduced. They are a generalization of torus actions, and in \cite{CheegerGromov1986} it was proven that a compact manifold equipped with Riemannian metric that is invariant under the $\F$-structure collapses with bounded curvature. 

Given a compact manifold $M$, an \emph{$\F$-structure} is a partition $\fol$ of $M$ together with an open cover $\{U_i\}_{i\in \Lambda}$ of $M$, and for each $i\in \Lambda$, there exists a finite normal covering map $\pi_i\colon \widetilde{U}_i\to U_i$ with covering transformation group $G_i$, an effective smooth $T^{k_i}$-action on $\widetilde{U}_i$ such that the there exists a map $\psi_i\colon G_i\to \mathrm{Aut}(T^{k_i})$ with $g_i(\xi_i\cdot \tilde{p})) = \psi(g_i)(\xi_i)\cdot \tilde{p}_i$ for $g_i\in G_i$, $\xi_i\in T^{k_i}$, $\tilde{p}\in \widetilde{U}_i$, and the following compatibility conditions for $U_{ij}:= U_i\cap U_j\neq \varnothing$:

\begin{enumerate}[(i)]
\item The submanifold $\pi_i^{-1}(U_{ij})$ is $T^{k_i}$-invariant. 
\item For the pullback $V_{ij}$ given by the following diagram
\begin{center}
\begin{tikzcd}
 V_{ij} \arrow[r]\arrow[d] & \pi^{-1}_j(U_{ij}) \arrow[d]\\
  \pi^{-1}_i(U_{ij})\arrow[r] & U_{ij}
\end{tikzcd}
\end{center}
we have finite covers $\widetilde{T}^{k_i}$ of $T^{k_i}$ and $\widetilde{T}^{k_j}$ of $T^{k_j}$ acting effectively on $V_{ij}$ and commuting with each other.
\end{enumerate}

A \emph{polarization} $\Pp$ of an $\F$-structure consist of a collection of connected subgroups $H_i\subset T^{k_i}$ such that the following holds:

\begin{enumerate}[(i)]
\item $H_i$ is invariant by the action of $\psi_{i}(G_{i})$.
\item For $p\in U_{ij}$, for any lifts $\tilde{p}_i\in \widetilde{U}_{i}$, $\tilde{p}_j\in \widetilde{U}_{j}$  either the image $\pi_{i}(H_{i}(\tilde{p}_i))$ contains $\pi_{j}(H_{j}(\tilde{p}_j))$ or viceversa (observe that this does not depend of the choices of $\tilde{p}_{i}$ and $\tilde{p}_{j}$).
\item The action of $H_i$ is locally free, i.e. $\dim(\pi_{i}(H_i(\tilde{p})))=\dim(H_i(\tilde{p})) = \dim(H_{i})$.
\end{enumerate}

We say that a polarization $\Pp$ is a \emph{pure polarization} if $\dim(H_{i}) = \dim(H_{j})$, i.e. for any $p\in U_{ij}$ and for any lifts $\tilde{p}_i\in \widetilde{U}_{i}$, $\tilde{p}_j\in \widetilde{U}_{j}$ we have $\pi_{i}(H_{i}(\tilde{p}_i)) =\pi_{j}(H_{j}(\tilde{p}_j))$.

An $\F$-structure is \emph{pure} if for any $p\in U_{ij}$ and for any lifts $\tilde{p}_i\in \widetilde{U}_{i}$, $\tilde{p}_j\in \widetilde{U}_{j}$ we have $\pi_{i}(T^{k_{i}}(\tilde{p}_i)) =\pi_{j}(T^{k_{j}}(\tilde{p}_j))$. An $\F$-structure has \emph{positive dimension} if the dimension of all the $T^{k_i}$-orbits is positive (i.e. the $T^{k_i}$-actions do not have fixed points).

\begin{example}
Any manifold $M$ equipped with smooth effective torus action gives an $\F$-structure, which may not be polarized, but it is pure.
\end{example}

\begin{example}
Consider the action of $T^2$ on $\Sp^3\subset \C^2$, by 
\[
(e^{i\theta_1},e^{i\theta_1})\cdot (z_{1},z_{2}):= (e^{2\pi\theta_1}z_{1},e^{2\pi\theta_2}z_{2}).
\]
We can define a polarization as follows: we consider $\{U_i\}_{i\in \Lambda}$ and arbitrary open covering by $T^2$-invariant open subsets. We set $\widetilde{U}_{i}:=U_{i}$ and $H_i:=G_1\times G_2$, where $G_j= \{e\}$ if $U_i$ intersects the subsphere $\{z_{j}=0\}$. We point out that this polarization is not pure, since the dimension of $H_i$ depends on the point we are.
\end{example}

\begin{example}
For the previous example we can define a pure polarization by setting $H=\{(e^{i pt},e^{i qt})\mid t\in \R \}$, for $p,q\neq 0$. Observe that when the line in $\R^2$ passes through $(p,q)$ and the origin does not have a rational slope, then $H$ is dense in $T^2$.
\end{example}

\begin{example}
The Klein bottle can be obtained as two open Möbius bands, intersecting at tubular neighborhood of their boundary, and observe that the double cover of each Möbius band is a cylinder $I\times \Sp^1$. Moreover, the circle action by rotation on the circle factor on each cylinder induces a foliation by circles on the Möbius band. Over the intersection of the Möbius bands, the actions of the circles agree with the rotation of the cylinder. Thus we get an $\F$-structure, which is a pure $\F$-structure not given by a torus action.
\end{example}

Given $M$ a manifold equipped with an $\F$-structure, a Riemannian metric $\g$ is \emph{$\F$-invariant} if the $T^{k_i}$-actions are by isometries with respect to $\widetilde{\g}_i$, the $\pi_{i}$-lifts of $\g|_{U_i}$. Such a  metric always exists for an $\F$-structure by the following theorem.

\begin{thm}[Lemma 1.3 in \cite{CheegerGromov1986}]
Let $M$ be a smooth manifold equipped with an $\F$-structure. Then there exists an $\F$-invariant Riemannian metric $\g$ on $M$. 
\end{thm}

Moreover, the quotient space $M^\ast$ of the partition induced by the $\F$-structure admits a metric $d^\ast$ induced by $\g$, making the quotient map $\pi\colon M\to M^\ast$ into a \emph{submetry}, i.e. for any $\varepsilon>0$ and $p\in M$ we have $\pi(B_{\varepsilon}(p))=B_{\varepsilon}(\pi(p))$.

In \cite{CheegerGromov1986} it was shown that the presence of $\F$-structure of positive dimension on a smooth  manifold $M$ with an $\F$-invariant Riemannian metric $\g$, allows us to collapse $(M,\g)$ with bounded curvature.

\begin{thm}[Theorem 4.1 in \cite{CheegerGromov1986}]\th\label{T: collapse of F-structures}
Let $(M,\g)$ be a compact Riemannian manifold. Assume that $M$ has an $\F$-structure of positive dimension and that $\g$ is $\F$-invariant. Then $(M,\g)$ collapses with bounded curvature
\end{thm}

In the general case, the diameter of the collapse given by \th\ref{T: collapse of F-structures} grows to infinity, and thus the collapsing sequence $(M,\g_{\delta})$ does not converge in the Gromov-Hausdorff sense to $(M^\ast,d^\ast)$.  However, in the case when the $\F$-structure admits a pure polarization $\Pp$, then we can also keep the diameter bounded.

\begin{thm}[Theorem 2.1 in \cite{CheegerGromov1986}]
Let $(M,\g)$ be a compact Riemannian manifold. Assume that $M$ has an $\F$-structure of positive dimension such that $\g$ is $\F$-invariant and there exists a pure polarization $\Pp$ of the $\F$-structure. Then $(M,\g)$ collapses with bounded curvature and bounded diameter, and the collapsing sequence converges in the Gromov-Hausdorff sense to $(M^\ast,d^\ast)$.
\end{thm}

In a later work \cite{CheegerFukayaGromov1992}, expanding on the ideas in \cite{Fukaya1987,Fukaya1989}, it was proven that a more general structure, the so-called $\Ns$-structure controls the collapse. In general an $\Ns$-structure induces an $\F$-structure, and in the case when $M$ has finite fundamental group the concepts of $\Ns$-structures and $\F$-structures agree, in the sense that an $\Ns$-structure is an $\F$-structure and viceversa. 

\begin{thm}[Theorems 1.3 and 1.7 in \cite{CheegerFukayaGromov1992}]\th\label{T: Existence of N structures on}
For $m>2$ and we denote by $\mathfrak{M}(m)$ the class of $m$-dimensional compact connected Riemannian manifolds $(M,\g)$ with $|\Sec(\g)|\leq 1$, finite fundamental group.

Given $\varepsilon>0$, there exists a constant $v=v(m,D,\varepsilon)>0$ such that given $(M,\g)\in \mathfrak{M}(m)$ with $\Vol(\g)< v$, then $M$ admits a pure $\F$-structure $\widetilde{\eta}_\varepsilon\colon F(M) \to B$ of positive rank. Moreover, 
\begin{enumerate}[(a)]
\item there is a smooth metric $\g^\varepsilon$ on $M$, which is $\F$-invariant for the $\F$-structure $\widetilde{\eta}_\varepsilon$,
\item the fibers of $\widetilde{\eta}_\varepsilon$ have diameter less than $\varepsilon$ with respect to $g^\varepsilon$, 
\item and
\[
e^{-\varepsilon} \g< \g^\varepsilon < e^{\varepsilon}\g.
\]
\end{enumerate} 
\end{thm}

Moreover, the metric $\g^\varepsilon$ can be chosen so that the sectional curvatures are close to the ones of $\g$ by the following theorem.

\begin{thm}[Theorem 2.1 in \cite{Rong1996}]\th\label{T: Rong bounds sectional curvature}
Let the assumptions be as in \th\ref{T: Existence of N structures on}. Then the nearby metric $\g^\varepsilon$ can  be chosen to satisfy in addition
\[
	\min \Sec(\g) -\varepsilon \leq \Sec(\g^\varepsilon) \leq\max \Sec(\g)+\varepsilon.
\]
\end{thm}

In particular the metric $\g^\varepsilon$ in \th\ref{T: Existence of N structures on}  induces an inner metric $d_{\g^\varepsilon}^\ast$ on the  quotient spaces of the $\F$-structures  making them into Alexandrov spaces of curvature at least $\min \Sec(\g)-\varepsilon$. 

\begin{remark}
When $M$ is simply-connected a pure $\F$-structure is given by a torus action $\mu_\varepsilon\colon T^k\times M\to M$ (see \cite[Lemma 1.4 and Lemma 3.2]{Rong1996}), and this action is by isometries with respect to the Riemannian metric $\g^\varepsilon$.
\end{remark}


\section{Singular Riemannian foliations}\label{Singular Riemannian foliations}

In this section we present the formal definition of a singular Riemannian foliation and define some properties of these foliations.

\begin{definition}\th\label{D: Singular Riemannian foliation}
A \emph{singular Riemannian foliation} $\fol=\{L_p\mid p\in M\}$ on a Riemannian manifold $(M,g)$ is a partition of $M$ into injectively immersed submanifolds $L_p$, called leaves, such that it is:
\begin{enumerate}[(i)]
\item a smooth singular foliation: There exists a family $\{X_\alpha\}$ of smooth vector fields of $M$ such that for any $p\in M$ with $p\in L_p$ we have $\mathrm{Span}\{X_\alpha(p)\} = T_p L_p$;
\item a transnormal system: For any geodesic $\gamma\colon [0,1]\to M$ with $\gamma'(0)\perp T_{\gamma(0)} L_{\gamma(0)}$, we have that $\gamma'(t)\perp T_{\gamma(t)} L_{\gamma(t)}$.
\end{enumerate}
\end{definition}

The foliation $\fol$ is called regular if the dimension of the leaves is constant. We say that a singular Riemannian foliation is \emph{closed} if all the leaves are closed manifolds (i.e. compact without boundary); in this case the leaves are embedded submanifolds. A subset $U\subset M$ is called \emph{saturated} if for any $p\in U$ we have $L_p\subset U$. 

The foliation $\fol=\{L_p=\{p\}\mid p\in M\}$ is called the \emph{trivial point-leaves foliation}, and the foliation $\fol=\{L_p=M\mid p\in M\}$ is called the \emph{trivial single-leaf foliation}. Observe that both foliations are singular Riemannian foliations with respect to any Riemannian metric $\g$ on $M$.

The \emph{leaf space} $M/\fol$ of a singular Riemannian foliation $(M,\fol,\g)$ is the quotient space induced by the partition $\fol$, equipped with the quotient topology, i.e. the finest topology on $M/\fol$ making the projection $\pi\colon M\to M/\fol$ continuous. When the foliation $\fol$ is closed and $(M,\g)$ is complete, the Riemannian metric $\g$ induces a metric $d^\ast$ which induced the quotient topology and such that $(M/\fol,d^\ast)$ is locally an Alexandrov space \cite{LytchakThobergsson2010}.

\subsection{Infinitesimally foliation}

Let $(M,\g)$ be a compact Riemannian manifold equipped with $\fol$ a closed singular Riemannian foliation. Fix a point $p\in M$ and take $0<\varepsilon<\mathrm{injrad}(\g)(p)$. The \emph{normal disk tangent disk at $p$ of radius $\varepsilon$} is
\[
\D^\perp_p (\varepsilon) = \{v\in T_p M\mid \|v\|^2_{\g}\leq \varepsilon\}.
\]
For $v\in \D^\perp_p (\varepsilon)$, taking $q:= \exp_p(v)\in L_{\exp_p(v)}$ by \cite{Molino} we have that $d(L_p,L_q) = d(L_p,q)=\|v\|_{\g}$. This means that a sufficiently close leaf $L_q$ to $L_p$ is contained in the boundary of a $\|v\|_{\g}$-tubular neighborhood of $L_p$. When we consider the connected components of the intersections of leaves of $\fol$ with $\exp_p(\D^\perp(\varepsilon))$, then by \cite{Molino} the preimages under $\exp_p$ of the connected components induce a smooth foliation $\fol_p(\varepsilon)$ on $\D^\perp_p(\varepsilon)$. Moreover, with respect to the (restricted) Euclidean inner product $\g_p$ on $T_p M$ the foliation  $\fol_p(\varepsilon)$ is a singular Riemannian foliation. Given $0<\varepsilon'\leq \varepsilon$, we can consider  the homothety $h_{\varepsilon/\varepsilon'}\colon T_p M\to T_p M$ defined as $h_{\varepsilon/\varepsilon'}(w) := (\varepsilon/\varepsilon') w$. The homothety $h_{\varepsilon/\varepsilon'}$ is a foliated diffeomorphism between  $(\D^\perp_p(\varepsilon'),\fol_p(\varepsilon'))$ and $(\D^\perp_p,\fol_p(\varepsilon))$. This implies that the foliation $\fol_p(\varepsilon)$ does not depend on the scale $\varepsilon$ we have fixed. Thus for any $\varepsilon'$ we can use the homothety $h_{\varepsilon/\varepsilon'}\colon \D^\perp_p(\varepsilon')\to \D^\perp_p(\varepsilon)$ to define $\fol_p(\varepsilon'):=h_{\varepsilon/\varepsilon'}^{-1}(\fol_p(\varepsilon)) =\{h_{\varepsilon/\varepsilon'}^{-1}(\mathcal{L}_{v}\mid v \in  \D^\perp_p(\varepsilon)\}$, and it is such that for $\varepsilon'\leq \varepsilon$ we have $\fol_p(\varepsilon') = h_{\varepsilon/\varepsilon'}^{-1}(\fol_p(\varepsilon))$. Thus we may assume without loss of generality that $\fol_p$ is defined on $\nu_p(L_p)$, the normal space of $L_p$ at $p$. We call this foliation \emph{the infinitesimal foliation at $p$}, and by \cite{Molino} it is also a singular Riemannian foliation with respect to the Euclidean Riemannian metric $\g_p|_{\nu_p(L_p)}$.

We observe that the leaves of $\fol_p(\varepsilon)$ are contained in spheres in $\D^\perp_p(\varepsilon)$. That is, for $v\in \D^\perp_p(\varepsilon)$ the subset $\Sp^\perp_p(\|v\|_{\g}) := \{w\in \D^\perp_p(\varepsilon)\mid \|w\|^2_{\g}= \|v\|^2_{\g}\}$ is saturated by $\fol_p(\varepsilon)$. Since we have established that $\fol_p^\perp(\varepsilon)$ is independent of $\varepsilon$,  we may assume without loss of generalization that $\fol_p$ restricted to the boundary sphere of the unit disk $\Sp^\perp_p=\partial \D^\perp(1)\subset \nu_p(L_p)$ induces a singular foliation,  also called the infinitesimal foliation, and which we also denote by $\fol_p$. Moreover by \cite{Molino}, the foliation $(\Sp^\perp_p,\fol_p)$ is a singular Riemannian foliation with respect to the round metric on $\Sp^\perp_p$ induced from the Euclidean metric $\g_p|_{\nu_p(L_p)}$.


\section{Collapse of singular Riemannian foliations with flat leaves}\label{S: Collapse of singular Riemannian foliations with flat leaves}

In this section we consider $(M,\g)$ a complete Riemannian manifold equipped with a closed singular Riemannian foliation $\fol$ with connected leaves such that for any $L\in \fol$, the submanifold $(L,\g|_{L})$ is a flat Riemannian manifold. We refer to such a foliation as a \emph{foliation with flat leaves}, or a \emph{$B$-foliation}, where the $B$ stands for the leaves being Bieberbach manifolds (see \cite{Corro2019,Corro2024}). Thus by the Cartan-Hadamard Theorem, for any leaf $L$ in a foliation with flat leaves we have that its universal cover is contractible. This implies that each leaf is aspherical, and thus this foliation is an $A$-foliation. Moreover by \cite{Bieberbach} the leaves are isometric to a so-called Bieberbach manifold. That is, the quotient of a Euclidean space by the free action of a discrete group of   

Since we can collapse  closed flat manifolds with bounded curvature and diameter is natural to ask if it possible to collapse a continuous parametrization of such manifolds, i.e. a foliation, by ``shrinking'' the directions tangent to the leaves while keeping the curvature bounded, including the curvature of the mixed planes spanned by vectors tangent to the leaves of the foliation, and the horizontal vectors perpendicular to the leaves of the foliation. 

This is possible for a regular flat foliation. Here we present a different proof of \cite[Theorem B]{Corro2024}.\\

\begin{duplicate}[\ref{T: Collapse of regular foliation}]
Consider $(M, \fol, \g)$ a regular non-trivial foliation with flat leaves on a compact Riemannian manifold. For $0<\delta\leq 1$ and $p\in M$ given we consider $T_p M= T_pL_p \oplus \nu_p (L_p)$ and set 
\[
\g_\delta(p) = \delta^2 (\g(p)|_{T_p L_p})\oplus (\g(p)|_{\nu_p(L_p)}).
\]
Then $(M,\g_\delta)$ collapses with bounded curvature and uniformly bounded diameter to the leaf space $(M/\fol,d^\ast)$.
\end{duplicate}

\begin{proof}
Recall that $(M/\fol,d^{\ast})$ is a smooth orbifold, and that $d^\ast$ is induced by a smooth orbifold Riemannian metric $\g^{\ast}$. Moreover, over the principal stratum of $\fol$ the quotient map is a Riemannian submersion $\pi\colon (M_{\prin},\g)\to (M^{\ast}_{\prin},\g^\ast)$.

From \cite[Section 2.1]{GromollWalschap} we have the following sectional curvature identities over $M_\prin$, for $V,W\in T\fol$ and $X,Y\in (T\fol)^{\perp}$:
\begin{align}
\Sec(\g_\delta)(X,Y) &=(1-\delta)\Sec_{M/\fol}(\pi_{\ast}X,\pi_{\ast}Y)+\delta \Sec(\g)(X,Y),\\
\Sec(\g_\delta)(X,V) &= \Sec(\g)(X,V)- \frac{(1-\delta)}{\|X\|_{\g}^2\|V\|_{\g}^2}\|A^{\ast}_X V\|^2_{\g},\\
\Sec(\g_\delta)(V,W) &= \frac{(1-\delta)}{\delta^2} \Sec(\g|_{T\fol})(V,W)+\Sec(\g)(V,W).
\end{align}
From this we see that over the principal part the sectional curvature remains uniformly bounded independent of $\delta$ when $\Sec(\g|_{T\fol})\equiv 0$. Since the sectional curvatures of $(M,\g)$ and $(M/\fol,\g^\ast)$ are continuous, the conclusion about the uniform bounds of $\Sec(\g_{\delta})$ holds.
\end{proof}

This particular construction allows us to compare in \cite{Corro2024} regular closed  foliations with $\F$-structures in the case when $M$ has finite fundamental group, and to conclude that they are induced by torus actions when $M$ is simply connected.

In the case when the foliation is not regular but the singular strata are isolated, by doing a procedure similar to the case of polar $\F$-structures in \cite{CheegerGromov1986} we can collapse the foliated manifold with bounded curvature but not with bounded diameter.\\

\begin{duplicate}[\ref{MT: Collapsing general flat foliations}]
Consider $(M, \fol, \g)$ a  closed singular Riemannian foliation with flat leaves on a compact Riemannian manifold, such that all singular strata are isolated from one another. Then we can collapse $(M,\g)$  with bounded curvature, in a way that also the volume of the collapsing sequence goes to $0$.
\end{duplicate}
\vspace*{1em}

Let $(M,\g,\fol)$ be a  non-trivial closed singular $m$-dimensional Riemannian foliation with flat leaves on a compact Riemannian manifold. Given a  leaf $L_p\in \fol$ of dimension $\ell <m$. Consider $B$ the connected component of the dimension stratum $\Sigma_\ell$ containing $L$ and $\pi\colon U\to B$ the closest point projection (this map is well defined when we choose the radius of the tubular neighborhood sufficiently small). Then via the exponential map and the homothetic transformations Lemma (see for example \cite[Theorem 2.1]{AlexandrinoRadeschi2016}), we can identify $p\colon U\to B$ with the normal bundle $\nu(B)\to B$ in a foliated fashion.

As shown in \cite{AlexandrinoRadeschi2017} there exists on $TM|_U$ three distributions $\mathcal{T}$, $\mathcal{N}$ and $\mathcal{K}$, where $\mathcal{K} = \ker(p_\ast)$. Moreover, $\mathcal{T}|_{L_p} = TL_p$, and $(\mathcal{T}\oplus \mathcal{N})|_{B}=TB$. With these distributions we can define a local chart system on $U$.

\begin{proposition}[Proposition 6.1. in \cite{AlexandrinoRadeschi2017}]\th\label{P: Local foliated coordinates of linearized-foliation}
Around any point $p\in B$ there is a neighborhood $W$ of $p$ in $B$
such that $(p^{-1}(W), \fol^{\ell}|_{p^{-1}(W)})$ is foliated diffeomorphic to a product
\[
(\D^{\ell} \times \D^{n-\ell} \times U_p, \D^\ell \times \{\mathrm{pts.}\} \times \fol^\ell_p)
\]
where $\ell = \dim(\fol|_B)$ and $n = \dim (B)$.
\end{proposition}

\begin{remark}
From the proof this this Proposition one infers that the tangent spaces corresponding to the $\D^\ell$-factor is spanned by $\mathcal{T}$. This implies that $\mathcal{T}$ is integrable, and induces a foliation $\fol'_L\subset \fol$, and $\fol_L'|_B=\fol|_B$. The factor $\D^\ell\times \D^{m-\ell}\times \{0\}$ is mapped to $B$. Thus the tangent spaces of the $\D^{m-\ell}$-factor correspond  to $\mathcal{N}$. From these observations we get the following corollary.
\end{remark}

\begin{cor}\th\label{C: polarization foliation}
Let $(M,\g,\fol)$ be a  non-trivial closed singular $m$-dimensional Riemannian foliation. For a given leaf $L\in \fol$ of dimension $\ell$, there exists a tubular neighborhood $U$ of $B$, the connected component of the the dimensional stratum $\Sigma_\ell$ containing $L$, such that it admits a regular foliation $\fol_L'\subset \fol$ and $\fol_L'|_{B}=\fol_B$.
\end{cor}

\begin{remark}\th\label{R: Homoeomorphism type of polarization}
The leaves of $\fol_L'$ are diffeomorphic to $L$ by construction. Moreover, we have that for any $q\in B$ we have $L_q=L_q'$, and for  any $\tilde{q}\in U$ we have that $L_{\tilde{q}}'\subset L_{\tilde{q}}$. Actually we have that $\mathcal{T}_{\tilde{q}}\subset T_{\tilde{q}} L_{\tilde{q}}$. 
\end{remark}

\begin{remark}
In general given a finite covering $\{U_i\}_{i=1,\ldots,N}$ of $M$, when we denote $\fol_i$ the regular foliation on $U_i$ given by \th\ref{C: polarization foliation}, it is not clear that for a point $w\in U_i\cap U_j$  in a non-empty intersection, the leaf $L_{w,i}\in \fol_i$ contains or is contained in the leaf $L_{w,j}\in \fol_j$. That is, we cannot guarantee that the regular distributions $\mathcal{T}_i$ on $U_i$ and $\mathcal{T}_j$ on $U_j$ agree over the non-empty intersections. 
\end{remark}

Nonetheless, such an obstruction only happens when the strata are not closed. This can also be characterized with the help of the following results.

Before we recall a result from \cite{RadeschiThesis}, recall that given two round spheres $\Sp^k\subset \R^{k+1}$ and $\Sp^{n-k-1}\subset\R^{n-k}$ we can parametrize the join $\R^k\oplus\R^{n-k-1}\supset\Sp^k\star \Sp^{n-k-1}\cong \Sp^n$ as:
\[
\Sp^k\star \Sp^{n-k-1}:= \{(\cos(t)v,\sin(t)w)\mid t\in [0,\pi],\  v\in \Sp^k,\  w\in\Sp^{n-k-1} \}.
\]

\begin{proposition}[Proposition 3.0.5. in \cite{RadeschiThesis}]\th\label{P: Splitting of infinitesimal foliation}
Let $(\Sp^n,\fol,\sigma)$ be a singular Riemannian foliation on the round sphere. Then the $0$-dimensional stratum $\Sigma_0\subset \Sp^n$ is a totally geodesic subsphere $\Sp^k\subset \Sp^n$, and the foliation $\fol$ decomposes as the joint of two singular Riemannian foliations on round spheres:
\[
(\Sp^k,\fol_0)\star (\Sp^{n-k-1},\fol_1)
\]
under the following identification:
\[
\fol \ni\mathcal{L}_{(\cos(t)v,\sin(t)w)} = \{(\cos(t)v',\sin(t)w')\mid v'\in \mathcal{L}_v\in \fol_0\ w'\mathcal{L}_w \in \fol_1\}.
\]
The foliation $\fol_0$ has only $0$-dimensional leaves, and all the leaves in $\fol_1$ have positive dimension.
\end{proposition}

\begin{remark}
The splitting in \th\ref{P: Splitting of infinitesimal foliation} is up to a foliated isometry.
\end{remark}

Since the infinitesimal foliation $(\nu_p(L_p),\fol_p)$  of a closed foliation $(M,\fol)$ at a point $p\in M$ is given by the cone of the infinitesimal foliation on the normal sphere $(\Sp_p^\perp,\fol_p)$, then from \th\ref{P: Splitting of infinitesimal foliation} we get the following property.

\begin{cor}\th\label{C: splitting of full infinitesimal foliation}
Let $(M,\fol,\g)$ be a closed singular Riemannian foliation on an $m$-di\-men\-sio\-nal manifold. For $p\in M$ with $\dim(L_p)=\ell$ and such that the connected component of $\Sigma_\ell$ has dimension $n$, we have that the infinitesimal foliation $(\nu_p(L_p),\fol_p,(\g_p)|_{\nu_p(L_p)})$ is foliated isometric to a product $(\R^{n-\ell},\fol_0)\times (\R^{m-n},\fol_1)$ where $n$ is the dimension of the connected component of the dimensional stratum that contains $L_p$, $\fol_0$ is the trivial point-leaves foliation, i.e. each leaf of $\fol_0$ consists of a single point, and $\fol_1$ contains leaves of positive dimension, except at the origin $0\in\R^{m-k}$, where the leaf consists of only the origin. 
\end{cor}

\begin{remark}
The $\R^n$-factor corresponds to normal directions to $L_p$ at $p$ which are tangent to the stratum containing $L_p$.
\end{remark}

\begin{remark}
Let $(M,\fol,\g)$ be a closed singular Riemannian foliation, and fix $p\in M$. We consider the foliated decomposition of infinitesimal foliation $(\nu_p(L_p),\fol_p) \cong (\R^n,\fol_0)\times (\R^{n-\ell},\fol_1)$ given by \th\ref{C: splitting of full infinitesimal foliation}. We point out that when the foliation $\fol_1$ is not regular, then $p$ is in the closure of some singular stratum different to the one that $p$ belongs to. Conversely, if a singular stratum is not closed, then for any point $p$ in the closure of the singular stratum, we have that the foliation $\fol_1$ is not regular. Thus a  closed singular Riemannian foliation has all singular strata closed, and thus isolated from one another, if and only if, for any point, the foliation $\fol_1$ is regular. 
\end{remark}

\begin{lemma}\th\label{L: existence of cover with compatible regular foliation}
Let $(M,\fol,\g)$ be a closed singular Riemannian foliation on a compact manifold, such that for each point $p\in M\setminus M_\reg$ the following holds for the infinitesimal foliation: the foliation $\fol_1$ in \th\ref{C: splitting of full infinitesimal foliation} is a regular foliation. Then there exists a finite open cover $\{U_i\}_{i=0,\ldots, N}$ such that for each index $i$, there exists a regular foliation $\fol_i$ on $U_i$, and for $w\in U_i\cap U_j$, we have that the leaf $L_{w,i}\in \fol_i$ is either contained in, or contains $L_{w,j}\in \fol_j$.
\end{lemma}

\begin{proof}
From our hypothesis we see that each singular stratum is isolated from other singular strata. If not, then we would be able to find a singular stratum and a point in this stratum , such that  the infinitesimal foliation  at this point would have singular leaves which do not correspond to directions tangent to the stratum.

For each connected component of a singular stratum, by \th\ref{C: polarization foliation} we can find a tubular neighborhood $U_i$ around it with a regular foliation. By taking the tubular neighborhoods to be small enough, we can guarantee that these foliated tubular neighborhoods $U_1,\ldots,U_N$ are pairwise disjoint. 

Now assume that the smallest tubular neighborhood has radius  $r>0$. Let $U_0\subset M_\reg$ be the complement in $M$ of the union of the closed tubular neighborhoods of radius $r/2$ over the singular strata, and observe that $\fol|_{U_0}$ induces a regular foliation.

Now we observe that since only $U_0$ intersects any other element in the open cover, the non-empty intersections are $U_i\cap U_0$. In this case by construction we have that for $w\in U_i\cap U_0$ we have $L_{w,i}\subset L_w = L_{0,w}$. 
\end{proof}

\begin{lemma}\th\label{L: Polarization of B-foliations}
Let $(M,\fol,\g)$ be a  non-trivial closed Singular $m$-dimensional Riemannian foliation with flat leaves on a compact Riemannian manifold. Given a  leaf $L\in \fol$ of dimension $\ell <m$, consider the tubular neighborhood $U$ of the strata containing $L_p$ with the regular foliation $\fol_L'\subset \fol$ given from \th\ref{C: polarization foliation}. Then for any $p\in L$, for the system of local coordinates $(x,y)=(x_1,\ldots,x_{\ell},y_1,\ldots,y_{n-\ell},z_1,\ldots,z_{m-n})$ in \th\ref{P: Local foliated coordinates of linearized-foliation}, we have that
the leaves of $\fol_L'$ consist of points where $y_j = \mathrm{constant}_j$ and $z_i=\mathrm{constant}_i$, and  we have
\[
g(x,y,z) = g(y,z) =\begin{pmatrix}
A(y,z) & B(y,z)\\
B(y,z) & \tilde{C}(y,z)
\end{pmatrix},
\]
where $A_{j,i} = g(\partial/\partial x_j,\partial/\partial x_i)$, $B_{j,i} = g(\partial/\partial x_j,\partial/\partial \alpha_i)$  and \linebreak$\tilde{C}_{j,i} = g(\partial/\partial \alpha_j,\partial/\partial \beta_i)$ for $\alpha=y,z$.
\end{lemma}

\begin{proof}
We set $B$ the connected component of the strata $\Sigma_\ell$ that contains $L$. And we consider $\tilde{W}$ a sufficiently tubular neighborhood of $L$, and set $W = \pi^{-1}(\tilde{W}\cap B)$, where $\pi\colon U\to B$ is the closest point projection.

Recall that for the  integrable regular distributions $\mathcal{T}$, $\mathcal{N}$ and $\mathcal{K}$ we have $\mathcal{N}(q)$ is $g(q)$-orthogonal to $\mathcal{K}(q)$ on $W$. Moreover $n:=\dim(\mathcal{N})=\dim(B)-k$, $k:=\dim(\mathcal{K}) =\dim(M)-\dim(B)$.

From these distributions we obtain by \th\ref{P: Local foliated coordinates of linearized-foliation} over a point $q\in W$ a local coordinate system $(\tilde{x}_1,\ldots,\tilde{x}_\ell,y_1,\ldots,y_{n-\ell},z_1,\ldots,z_k)$.

We also observe that since $\mathcal{N}$ is orthogonal to $\mathcal{K}$, then an integral curve $\gamma$ tangent to $\mathcal{N}$ gets projected to a curve which is equidistant to $B$. This means that $\mathcal{N}$ is normal to all leaves of $\fol$ in $W$.

We point out that for $q\in W\setminus (W\cap B)$ we have that with respect to the coordinates in \th\ref{P: Local foliated coordinates of linearized-foliation} around $q$ we have that  $(x,\mathrm{cte}_1,\mathrm{cte}_2)$ are contained in the Leaf $L_q$, where the constants $\mathrm{cte}_i$ are determined by $q$. Observe that locally the full leaf $L_q$ is determined by a graph over the variables $x$ and $z$, but some of the $z_j$-factors are constant. This implies that the foliated Riemannian metric does not depend on the variable $x$ since it a homogeneous metric on $L_q$. 
\end{proof}

\begin{lemma}
Consider $(M,\fol,\g)$ and $L\in\fol$ as in the previous Lemma. Assume that $L$ is diffeomorphic to an $\ell$-dimensional torus. Then the leaves of $\fol_L'$  are also flat.
\end{lemma}

\begin{proof}
We have that for $q\in U$, the leaf $L_q'\in \fol_L'$ is contained in the flat space $(L_q,\g_{TL_{q}})$. By \th\ref{R: Homoeomorphism type of polarization}, we have then a subtorus $L_q'$ of the flat manifold $L_q$. Assume that $k=\dim(L_q)\geq \ell$. Passing to the flat Riemannian universal cover $(\R^k,\widetilde{\g_{TL_{q}}})$ of $(L_q,\g_{TL_q}$, the submanifold $L_q'$ corresponds to a linear subspace $A_q$ being spanned by vectors with integer entries. This implies that the metric induced by $\g_{TL_{q}}$ on $L_q'$ is flat, since it is given by the metric $\widetilde{\g_{TL_{q}}}$ induced on the linear subspace $A_q$.
\end{proof}

\begin{lemma}
Let $(M,\fol,\g)$ be a closed singular Riemannian foliation as in \th\ref{L: existence of cover with compatible regular foliation}. Furthermore, assume that the leaves of the foliation are flat manifolds with respect to the induced metric. Let $\{U_i\}_{i=1,\ldots,N}$ be the open cover with regular foliations given by \th\ref{L: existence of cover with compatible regular foliation}, and  $p\in U_i$ be contained in a singular leaf. Then under the coordinates presented in \th\ref{L: Polarization of B-foliations} the regular leaves in $U_i\cap U_0$ are determined by, after a linear change of coordinates, by having some $z_j$-coordinates constant.
\end{lemma}

\begin{proof}
This follows from combining the proof of Lemma \ref{L: Polarization of B-foliations} with Lemma \ref{L: existence of cover with compatible regular foliation}.
\end{proof}

\begin{proof}[Proof of \th\ref{MT: Collapsing general flat foliations}]
We follow the same strategy as in \cite[Section 4]{CheegerGromov1986}.

Consider the open cover $\{U_i\}_{i=1,\ldots,N}$ of $M$ from \th\ref{L: existence of cover with compatible regular foliation}. Now recall that for $i\neq 0$ each $U_i$ is a tubular neighborhood of a connected component $B_i$ of singular stratum. Then for each $i=1,\ldots,N$ we can cover $U_i$ by tubular neighborhoods $\{\widetilde{U}_{i,\alpha}\}_{\alpha\in \Lambda_i}$ of leaves in $B_i$, where the radius of $\widetilde{U}_{i,\alpha}$ is sufficiently small, so that \cite[Theorem A]{MendesRadeschi2015} holds. 

On each $\widetilde{U}_{i,\alpha}$ the regular foliation $\fol_i$ of $U_i$ induces a regular foliation $\fol_{i,\alpha}$ whose leaves are contained in the leaves of $\fol$. 

We extend the union of $\{\widetilde{U}_{i,\alpha}\}_{\alpha\in \Lambda}$ to an open cover of $M$, by taking also tubular neighborhoods centered at regular leaves, whose radius is small enough so that \cite[Theorem A]{MendesRadeschi2015} holds.

From this we obtain a new finite cover $\{\overline{U}_j\}_{j=1,\ldots,\bar{N}}$ of $M$. Recall that by \cite[Lemma 3.6]{CorroFernandezPerales22} there exists a partition of unity $\{\phi_j\}_{j=1,\ldots,\bar{N}}$ subordinate to $\{\overline{U}_j\}_{j=1,\ldots,\bar{N}}$. We set $\rho_{i,\delta}=\delta^{\log(\phi_i)/\log(1/2)}$.

On $\overline{U}_j$ we  denote the regular subfoliations we have constructed by $\overline{\fol}_j$, and observe that we have chosen an arbitrary ordering of the open cover.  On $\overline{U}_1$ we decompose $\g|_{\overline{U}_1}= \g_1'\oplus \hh_1$, where $\g_1'$ is the restriction of $\g|_{\overline{U}_1}$ to $T\overline{\fol}_i$. We set
\[
\g_1 = \begin{cases}\rho_{1,\delta}^2\g_1'\oplus \hh_1, & \overline{U}_1\\
\log(\delta)^2\g & M\setminus \overline{U}_1,
\end{cases}
\]
and proceeding by induction  on $\overline{U}_{j+1}$ we write $\g_j = \g_{j+1}'\oplus \hh_{j+1}$ where  $\g_{j+1} = \g_j|_{T\overline{\fol}_j}$. We now set 
\[
\g_j = \begin{cases}\rho_{j+1,\delta}^2\g_{j+1}'\oplus \hh_{j+1}, & \overline{U}_{j+1}\\
\g_j & M\setminus \overline{U}_{j+1}.
\end{cases}
\]

The metric $\g_\delta:=\g_{\overline{N}}$ induces collapse with bounded curvature, and the proof is verbatim to the one of \cite[Section 4]{CheegerGromov1986}: For this we use the identities in \cite[Section 2.1]{GromollWalschap}, and that how our metrics are changing depends on the description in \th\ref{P: Local foliated coordinates of linearized-foliation} of the foliated charts. Moreover the same proof yields that $\mathrm{vol}(\g_\delta)(M)\leq C \delta^\ell |\log(\delta)|^m$, for some constant $C>0$, and $\ell=\dim(\fol)<m=\dim(M)$.
\end{proof}

As a corollary from \cite{CheegerGromov1990,CheegerFukayaGromov1992} we the following:\\

\begin{duplicateCOR}[\ref{MC: collapse controlled by torus action}]
Consider $(M, \fol, \g)$ a  closed singular Riemannian foliation with flat leaves on a compact Riemannian manifold with finite fundamental group, such that all singular strata are isolated from one another, and assume that $\lambda\leq \Sec(\g)\leq \Lambda$.  Let $\{\g_i\}_{i\in \N}$ be the collapsing sequence given by \th\ref{MT: Collapsing general flat foliations}. Then, given $\varepsilon>0$ there exists $N:=N(\varepsilon)\in \N$ such that for all $i\geq N$ there exists an $F$-structure $\fol_{\varepsilon,i}$ on $M$ and an $\fol_{\varepsilon,i}$-invariant Riemannian metric $\g_{\varepsilon,i}$ such that
\[
e^{-\varepsilon}\g_i<\g_{\varepsilon,i}< e^{\varepsilon}\g_i,
\]
and
\[
\lambda-\varepsilon\leq \Sec(\g_{\varepsilon,i})\leq \Lambda+\varepsilon.
\]
\end{duplicateCOR}

\begin{remark}
Now we point out that the same proof as \cite[Theorem A]{Corro2024} of rescaling the metrics from \th\ref{MC: collapse controlled by torus action} in the directions tangent to the leaf, to conclude that the foliation is homogeneous. This is because the collapsing sequence in \th\ref{MT: Collapsing general flat foliations} does not converge in the Gromov Hausdorff sense to the leaf space $(M/\fol,d^\ast)$ due to the fact that we need to introduce the factor $\log(\delta)$ to control the sectional curvature.
\end{remark}


\section{Lie groupoid action}\label{S: Lie groupoid actions}

\begin{definition}
A \emph{Lie groupoid $\G\rightrightarrows M$} consists of two smooth manifolds $\G$ and $M$ together with two submersions $s,t\colon \G\to M$, a smooth embedding $\1\colon M\to \G$, a diffeomorphism $i\colon \G\to M$, and a smooth map $\m\colon \G^{(2)} = \{(g,h)\in \G\times \G\mid s(g) = t(h)\}\to \G$, which satisfy the following:
\begin{enumerate}[i)]
\item For all $(g,h)\in \G^{(2)}$ it holds
\[
	s(\m(g,h)) = s(h).
\]
\item For all $(g,h)\in \G^{(2)}$ it holds
\[
	t(\m(g,h)) = t(g).
\]
\item For all $(\tilde{g},g),(g,h)\in \G^{(2)}$ it holds
\[
	\m(\m(\tilde{g},g),h) = \m(\tilde{g},\m(g,h)).
\]
\item For all $p\in M$ and all $g\in \G$ it holds
\[
	s(\1_p) = p = t(\1_p), \qquad \m(\1_{t(g)},g) = g = \m(g,\1_{s(g)}).
\]
\item For all $g\in \G$ it holds
\[
	\m(g,i(g)) = \1_{t(g)},\qquad \m(i(g),g) = \1_{s(g)}.
\]
\end{enumerate}
\end{definition}

Given a Lie groupoid $\G\rightrightarrows M$, there is a partition $\fol_\G$ of $M$ induced by the groupoid: Given $p\in M$ we set $L_p = t(s^{-1}(p))$, which is an embedded submanifold of $M$. 

Given a subset $S\subset M$ we say that it is \emph{saturated} if for any $p\in S$ we have $L_p\subset S$.

\begin{definition}[Restriction groupoid]
Let $\G\rightrightarrows M$ be a Lie groupoid, and let $S\subset M$ be a smooth submanifold $S\subset M$.  Since both $s$ and $t$ are submersions, we have that $s^{-1}(S)$ and $t^{-1}(S)$ are smooth submanifolds of $\G$. When $s^{-1}(S)\cap t^{-1}(S)$ is also a submanifold,  we can define the \emph{restriction groupoid} $\G_S\rightrightarrows S$ as follows: We set $\G_S = s^{-1}(S)\cap t^{-1}(S)$, with the structure maps given by setting  restricting the structure maps of $\G\rightrightarrows M$, to $\G_S$, $\G_S^{(2)}$, and $S$. For example, when  $U\subset M$ is an open subset,  $\G_U$ is an open subset of $\G$, and thus we can define the  restriction groupoid  $\G_U\rightrightarrows U$. Observe that a subset  $S\subset M$  is a saturated subset if and only if  $s^{-1}(S)=t^{-1}(S)$. Thus we can define the restriction groupoid $\G_S\rightrightarrows S$ for saturated submanifold $S$, and the codimension of $\G_S$ in $\G$ is the same as the codimension of $S$ in $M$. Any leaf $L\subset M$ of $\G\rightrightarrows M$ is a saturated submanifold.
\end{definition}

\begin{remark}
When we consider $S$ to be an orbit $L\subset M$ of the Lie groupoid $\G\rightrightarrows M$, then by \cite[Section 3.4]{delHoyo2013}, we get the restriction groupoid $\G_L\rightrightarrows L$.
\end{remark}

Given  a Lie groupoid  $\G\rightrightarrows M$,  a smooth manifold $P$ and  a smooth map $\alpha\colon P\to M$, we define a \emph{left action of $\G\rightrightarrows M$ on $P$ along $\alpha$} is a map $\mu\colon \G\times_M P := \{(g,p)\in \G\times P\mid s(g) = \alpha(p)\}\to P$ such that: 
\begin{enumerate}[(i)]
\item For all $(g,p)\in \G\times_M P$ we have $\alpha(\mu(g,p)) = t(g)$.
\item For all $(g,\mu(h,p))$, $(h,p)\in \G\times_M P$ and $(g,h)\in \G^{(2)}$ we have $\mu(g,\mu(h,p)) = \mu(\m(g,h),p)$.
\item For all $x\in M$ with $x = \alpha(p)$, we have $\mu(\1_{x},p) = p$.
\end{enumerate}
A left action of $\G\rightrightarrows M$ on $P$ along $\alpha$ induces a Lie groupoid $\G\times_M P\rightrightarrows P$, with the following maps $\overline{s}(g,p) := p$, $\overline{t}(g,p) := \mu(g,p)$, $\overline\1_p := (\1_{\alpha(p)},p)$, $\overline{i}(g,p):=(i(g),\mu(g,p))$, and for $(g_2,g_1)\in \G^{(2)}$ we set $\overline{\m}((g_2,\mu(g_1,p)),(g_1,p)):= (\m(g_2,g_1),p)$.

We observe that any action $\mu$ of $\G\rightrightarrows M$ on $\alpha\colon P\to M$ realizes the elements  of $\G$ as symmetries of the family of fibers of the  map $\alpha$, i.e. for each $g\in \G$ we have a diffeomorphism $\mu_g\colon P \supset \alpha^{-1}(s(g))\to \alpha^{-1}(t(g))\subset P$ given by $\mu_g(p ) = \mu(g,p)$. 

In the case when $\alpha\colon P\to M$ is a vector bundle, we call a Lie groupoid action of $\G\rightrightarrows M$ on $P$ along $\alpha$ a \emph{left  representation of $\G\rightrightarrows M$} if  the action is such that for each $g\in \G$ the map $\mu_g\colon \alpha^{-1}(s(g))\to \alpha^{-1}(t(g))$ is a linear map. 

\begin{example}[Left Lie groupoid action]\th\label{E: Left Lie groupoid action}
Given a Lie groupoid $\G\rightrightarrows M$, we define the \emph{left Lie groupoid action} of $\G\rightrightarrows M$ on $t\colon \G\to M$, to be the map $\mu\colon \G\times_M \G\to M$ defined for $(g,h)\in \G^{(2)}$ as
\[
\mu(g,h) = \m(g,h).
\]
\end{example}

\begin{example}[Canonical Lie groupoid action]\th\label{E: canonical Lie groupoid action}
Given a Lie groupoid $\G\rightrightarrows M$, the \emph{canonical Lie groupoid action} of $\G\rightrightarrows M$ on $\Id_M\colon M\to M$ is given by the map $\mu\colon \G\times_M M\to M $ defined as
\[
\mu(g,s(g)) = t(g).
\]
\end{example}

\begin{remark}
Observe that in \th\ref{E: canonical Lie groupoid action} the orbits of $\G\times_M M\rightrightarrows M$ coincide with the orbits of $\G\rightrightarrows M$.
\end{remark}

A \emph{free} left Lie groupoid action $\mu$ of $\G\rightrightarrows M$ on $P$ along $\alpha\colon P\to M$ is one such that the Lie groupoid $\G\times_M P\rightrightarrows P$ has trivial isotropy groups. Moreover, it is called \emph{proper} if the map $\G\times_M P\to P\times P$ defined by $(g,p)\mapsto (\mu(g,p),p)$ is a proper map. We define the orbit space of the action, denoted by $P/\G$, as the quotient space of the partition  $\fol_{\G\times_M P}$ of $P$ induced by the orbits of the Lie groupoid $\G\times_M P\rightrightarrows P$, equipped with the quotient topology.

\subsection{Normal representation} We begin by pointing out that given a smooth group action $\mu\colon H\times M\to M$ of a Lie group $H$ on a smooth manifold $M$, we have a lift of the action to the tangent bundle of $M$, $\mu_\ast\colon H\times TM\to TM$. This action is given as follows: For $h\in H$ fixed, we can consider the diffeomorphism $\mu_h\colon M\to M$ defined as $\mu_h(p) = \mu(h,p)$. We then set $\mu_\ast(h,X) = (\mu_h)_\ast (X)$, for $h\in H$ and $X\in TM$. 

In contrast, for a Lie groupoid action $\G\rightrightarrows M$ on $\alpha\colon P\to M$ such a natural lift does not exist. But we have the so-called \emph{normal representation} of the action groupoid $\G\times_M P \rightrightarrows P$.

Given two Lie groupoids $\G\rightrightarrows M$ and $\G'\rightrightarrows M'$, we define a \emph{morphism of Lie groupoids} $\phi\colon (\G\rightrightarrows M)\to (\G'\rightrightarrows M')$ to be two smooth maps $\phi^{\ar}\colon \G\to \G'$ and $\phi^{\ob}\colon M\to M'$ such that for the structure maps $\{s,t,\1,i,\m\}$ of $\G\rightrightarrows M$ and $\{s',t',\1',i',\m'\}$ of $\G'\rightrightarrows M'$ we have $s'\circ \phi^{\ar} = \phi^{\ob}\circ s$, $t'\circ \phi^{\ar} = \phi^{\ob}\circ t$, $i'\circ \phi^{\ar} = \phi^{\ar}\circ i$, $\1'\circ\phi^{\ob} = \phi^{\ar}\circ \1$, and for $(g,h)\in \G^{(2)}$ we have $\mu'(\phi^{\ar}(g),\phi^{\ar}(h)) = \phi^{\ar}(\m(g,h))$.

An \emph{VB-groupoid} consists of two groupoids $\Gamma\rightrightarrows E$ and $\G\rightrightarrows M$ and a groupoid map $\phi\colon (\Gamma\rightrightarrows E) \to (\G\rightrightarrows M)$, such that the maps $\phi^{\ar}\colon \Gamma\to \G$ and $\phi^{\ob}\colon E\to M$ are vector bundles, and the structure maps of $\Gamma\rightrightarrows E$ are vector bundle maps. We define the \emph{core} of the VB-groupoid to be $C = \kernel(s\colon \Gamma\to E)|_M$, where we identify $M$ with the $0$-section of $\phi^{\ob}\colon E\to M$. By the following proposition, VB-groupoids are equivalent to linear representations of Lie groupoids.

\begin{proposition}[Proposition~3.5.5 in \cite{delHoyo2013}]\th\label{P: correspondence between VB-groupoids and representations}
There is a $1$--$1$ correspondence between linear representations of a Lie groupoid $G\rightrightarrows M$ and VB-groupoids $(\Gamma\rightrightarrows E) \to (\G\rightrightarrows M)$  with trivial core.
\end{proposition}

We now consider $\G\rightrightarrows M$ a Lie groupoid, and fix an orbit $L\subset M$. For the restriction groupoid $\G_{L}\rightrightarrows L$ the embedding $L\subset M$ induces an embedding $\G_L\subset \G$. With this embeddings we define the normal bundles $N(\G_L)\to \G_L$ and $N(L)\to L$ by setting $N(\G_L) = T\G|_{\G_L}/T\G_L$ and $N(L) = TM|_L/TL$. With this we obtain the following sequence of VB-groupoids
\[
(T\G_L)\rightrightarrows TL)\to (T\G|_{\G_L}\rightrightarrows TM|_{L})\to (N(\G_L)\rightrightarrows N(L)).
\]
Observe that as with the tangent groupoid of a Lie groupoid, we have a VB-groupoid 
\[
\pi\colon (N(\G_L)\rightrightarrows N(L))\to (\G_L\rightrightarrows L).
\] 
Moreover, since $\G_L\subset \G$ and $L\subset M$ have the same codimension, this implies that the bundles $N(\G_L)\to \G_L	$ and $N(L)\to L$ have the same rank.  From this it follows that the core of the VB-groupoid 	$\pi\colon (N(G_L)\rightrightarrows N(L))\to (\G_L\rightrightarrows L)$	has trivial core. 
From \th\ref{P: correspondence between VB-groupoids and representations} we have  a linear representation	$N(\mu)$ of $\G_L\rightrightarrows L$ on \linebreak$\pi_{N(L)}\colon N(L)\to L$, called the \emph{normal representation of $\G_L\rightrightarrows L$}, or the \emph{normal representation of $\G\rightrightarrows M$ around $L$}.	

We can also express the normal representation as follows (see \cite[p.~178]{delHoyo2013}). Consider $q\in L_p$, and $g\in s^{-1}(q)$. For $v\in N_q(L_p)$ we consider a curve $\gamma\colon I \to M$ with $\gamma(0) = q$, whose velocity $\gamma'(0)$ represents the class  $v$, i.e. $[\gamma'(0)] = v$ in $T_q M|_{L_p}/T_q L_p$. Let $\widetilde{\gamma}\colon I \to \G$ be a curve such that $\widetilde{\gamma}(0) = g$ and $s\circ \widetilde{\gamma} = \gamma$. Such a curve always exists, since $s\colon \G\to M$ is a submersion. Then we set $N(\mu)(g,v)$ to be $(t\circ \widetilde{\gamma})'(0)$, that is:
\[
N(\mu)(g,v) = N(\mu)\big(g,[D_g s(\gamma'(0))] \big) = [D_g t(\gamma'(0))].
\]

\subsection{Riemannian Lie groupoids and invariant actions}\label{S: Riemannian Lie groupoids} In this section we present the notions of Riemannian metrics associated to Lie groupoids presented in \cite{delHoyoFernandes2018}.

Given a Lie groupoid action $\mu$ of $\G\rightrightarrows M$ on $\alpha\colon P\to M$, we can consider $L\subset P$ an orbit of the action groupoid $\G\times_M P\rightrightarrows P$. Given a Riemannian metric $\eta^P$ on $P$, we can identify the total spaces of the  normal bundle $N(L)\to L$ with $\nu(L):= \{v\in TP|_{L}\mid\forall x\in TL,\: \eta^{P}(v,x) =0\}\subset TP$. 

\begin{definition}\th\label{D: mu-transversly invariant}
The Riemannian metric $\eta^P$ is called \emph{transversely $\mu$-invariant} if for the  normal representation $N(\mu)$ of the action groupoid $(\G\times_M P)_L \rightrightarrows L$ on $\nu(L)\to L$  is by isometries, i.e. for $(g,p)\in (G\times_M P)_L$ the map $N(\mu)_{(g,p)}\colon \nu_{p}(L) \to \nu_{\mu(g,p)}(L)$ is an isometry.
\end{definition}

A \emph{$0$-metric} on a Lie groupoid $\G\rightrightarrows M$ is a Riemannian metric $\eta^{(0)}$ on $M$, such that it is transversely $\mu$-invariant for $\mu$ the canonical action of the Lie groupoid (see \th\ref{E: canonical Lie groupoid action}).

Since $0$-metrics are only on the object space of a Lie groupoid, we can also consider Riemannian metrics on the space of arrows. A \emph{$1$-metric} on a Lie groupoid $\G\rightrightarrows M$ is a Riemannian metric $\eta^{(1)}$ on $\G$ which is transversely $\mu$-invariant for the left Lie groupoid action of $\G\rightrightarrows M$ on $t\colon \G\to M$ (see \th\ref{E: Left Lie groupoid action}), and such that the inversion map $i\colon \G\to \G$ is an isometry. 

\begin{remark}
An equivalent definition of a $1$-metric, is that $\eta^{(1)}$ is a Riemannian metric on $\G$ such that the map $s\colon \G\to M$ is a Riemannian submersion (see \cite[Example 2.9]{delHoyo2013}).
\end{remark}

\begin{remark}
Observe that a Lie groupoid $\G= G\rightrightarrows \{\ast\}$ is a Lie group, then any Riemannian metric on $G$ induces a $1$-metric on $\G$. Thus, we observe that $1$-metrics do not account for the Lie groupoid multiplication. To consider the multiplication, it is necessary to consider the space of composable arrows $\G^{(2)}$ of the Lie groupoid $\G\rightrightarrows M$.
\end{remark}

As done in \cite[Section~3.3]{delHoyoFernandes2018}, we can identify a pair of composable arrows $(g,h)$, and their multiplication $\m(g,h)$ with a diagram as follows:
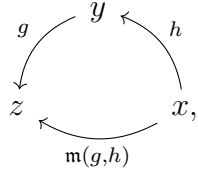
\begin{center}
\begin{tikzcd}[column sep = small]
	&y\arrow[bend right,swap]{dl}{g}&\\
	z & &x\arrow[bend left]{ll}{\m(g,h)}\arrow[bend right,swap]{ul}{h},
\end{tikzcd}
\end{center}
where $x = s(h) = s(\m(g,h))$, $y= t(h)=s(g)$ and $z= t(g) = t(\m(g,h))$. With this identification, we can define an action of the symmetric group $S_3$ on $\G^{(2)}$, by identifying the source $s(h)$ with $1$, the target $t(h)$ with $2$, and the target $t(g)$ with $(3)$. So for example, the element $(1,3)\cdot (g,h)$ corresponds to the diagram
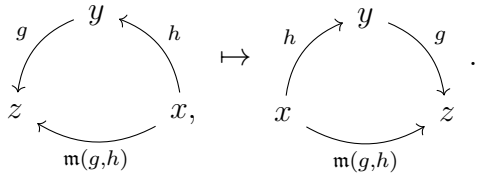
\[
\begin{tikzcd}[column sep = small]
	&y\arrow[bend right,swap]{dl}{g}&\\
	z & &x\arrow[bend left]{ll}{\m(g,h)}\arrow[bend right,swap]{ul}{h},
\end{tikzcd}\mapsto \begin{tikzcd}[column sep = small]
	&y\arrow[bend left]{dr}{g}&\\
	x\arrow[bend left]{ur}{h} \arrow[bend right,swap]{rr}{\m(g,h)} & &z
\end{tikzcd}.
\]
That is $(1,3)\cdot (g,h) = (i(h),i(g))$ (see \cite[Remark 3.13]{delHoyoFernandes2018} for an alternative definition). 

We have also three commuting groupoid actions $\mu_1$, $\mu_2$ and $\mu_3$ of $(\G\rightrightarrows M)$ on $\alpha_i\colon \G^{(2)}\to M$, $i\in \{1,2,3\}$, defined as follows:
\begin{enumerate}[(i)]
\item
\[
\mu_1\big(k,(g,h)\big) = \big(\m(k,g),h\big)\quad \mbox{and}\quad \alpha_1(g,h) = t(g),
\]
\item
\[
\mu_2\big(k,(g,h)\big) = \big(\m(g,i(k)),\m(k,h)\big)\quad \mbox{and}\quad \alpha_2(g,h) = t(h),
\]
\item
\[
\mu_3\big(k,(g,h)\big) = \big(\m(k,g),h\big)\quad \mbox{and}\quad \alpha_2(g,h) = s(h),
\]
\end{enumerate}

\begin{remark}
Observe that the action of $S_3$ on $\G^{(2)}$ interchanges the actions $\mu_i$ of $\G\rightrightarrows M$ on $\alpha_i\colon\G^{(2)}\to M$. Moreover, the orbits of the actions $\mu_1$, $\mu_2$, and $\mu_3$, i.e. the orbits of the action groupoids $\G\times_M \G^{(2)}\rightrightarrows \G^{(2)}$, are the fibers of the maps $\proj_2,\m,\proj_1\colon \G\times \G\supset \G^{(2)}\to \G$ respectively.
\end{remark}

A Riemannian metric $\eta^{(2)}$ on $\G^{(2)}$ is called a \emph{$2$-metric} on $\G\rightrightarrows M$ if it is transversely invariant for the $\mu_1$-action of  $\G\rightrightarrows M$ on $\alpha_1\colon \G^{(2)}\to M$ and the action of $S_3$ is by isometries. The pair $(\G\rightrightarrows M, \eta^{(2)})$ is called a \emph{Riemannian groupoid}. In the literature, the term ``Riemannian groupoid'' also refers to a $1$-metric on the Lie groupoid $\G\rightrightarrows M$, cf.\ \cite{GallegoGualdraniHectorReventos1989,Glickenstein2008}.

\begin{remark}
From the fact that $S_3$ permutes the actions $\mu_1$, $\mu_2$ and $\mu_3$, if follows that the definition of a $2$-metric $\eta^{(2)}$ on $\G^{(2)}$ is equivalent to the following definitions:
\begin{itemize}[itemsep = 0.5em]
\item $\eta^{(2)}$ is transversely $\mu_i$-invariant for any of the actions $\mu_i$ of $\G\rightrightarrows M$ on $\alpha_i\colon \G^{(2)}\to M$, and the group $S_3$ acts by isometries.
\item $\eta^{(2)}$ makes any of the maps $\proj_2,\m,\proj_1\colon \G^{(2)}\to \G$ a Riemannian submersion, and the group $S_3$ acts by isometries.
\end{itemize}
\end{remark}

As before, a $2$-metric on $\G^{(2)}$ induces a $1$-metric on $\G$, and also a $0$-metric on $M$.

\begin{proposition}[Proposition 3.16 in \cite{delHoyoFernandes2018}]\th\label{P: 2-metrics induce 1-metrics}
Let $\G \rightrightarrows M$ be a Lie groupoid. A $2$-metric $\eta^{(2)}$ on $\G^{(2)}$ induces a $1$-metric $\eta^{(1)}$ on $\G$.
\end{proposition}

\begin{remark}\label{R: 2-metrics on Lie groups}
Given a Lie group $\G = G\rightrightarrows \{\ast\}$, a bi-invariant metric $\eta$ on $\G$ is a $1$-metric on $G$, coming from the $2$-metric $\eta\times \eta$ on $\G^{(2)} = G\times G$. Nonetheless there are $1$-metrics on $G\rightrightarrows \{\ast\}$ coming from $2$-metrics which are not bi-invariant metrics (see \cite[Example 4.4]{delHoyoFernandes2018}).
\end{remark}

\subsection{Hausdorff proper Lie groupoids and existence of 2-metr\-ics} As with  $1$- and $0$-metrics, there are sufficient conditions to guarantee the existence of a $2$-metric given a $1$-metric. To state this condition we consider proper Lie groupoids. A Lie groupoid $\G\rightrightarrows M$ is \emph{proper}, if the map $\rho\colon \G\to M\times M$ given by $\rho(g) = (t(g),s(g))$ is a proper map. 

\begin{thm}[Theorem 1 in \cite{delHoyoFernandes2018}]\th\label{T: existence of 2-metrics}
Any proper Lie groupoid $\G\rightrightarrows M$ admits a $2$-metric. Moreover, given a $1$-metric $\eta^{(1)}$ on $\G\rightrightarrows M$, there exists a $2$-metric $\eta^{(2)}$ on $\G\rightrightarrows M$ such that $\eta^{(1)}$ agrees with the induced Riemannian metric given by Proposition \ref{P: 2-metrics induce 1-metrics}.
\end{thm}

Moreover, given a proper Lie groupoid action $\mu$ of $\G\rightrightarrows M$ on $\alpha\colon P\to M$, and a Riemannian  metric $\eta$ on $P$, we can always give a new Riemannian metric $\tilde{\eta}$ on $P$ which is $\mu$-transversely invariant.

\begin{thm}[Proposition~4.11 in \cite{delHoyoFernandes2018}]\th\label{T: existence of G->M invariant metrics on P->M}
Given any Riemannian  metric $\eta$ on $P$, and a Lie groupoid action $\mu$ of a proper $\G\rightrightarrows M$ Lie groupoid on $\alpha\colon P\to M$, there exists Riemannian metric $\tilde{\eta}$ on $P$ which is $\mu$-transversely invariant. Moreover, in the case when $\eta$ is already $\mu$-invariant, then $\eta$ and $\tilde{\eta}$ agree on the directions normal to the orbits of the action groupoid $\G\times_M P\rightrightarrows P$.
\end{thm}

\begin{remark}
We consider a Lie group $G$ as the Lie groupoid $G\rightrightarrows \{\ast\}$,  $(P,\eta)$ any Riemannian manifold, and a proper Lie group action $\mu\colon G\times P\to P$.  Then we have a proper Lie groupoid action of $G\rightrightarrows \{\ast\}$ along $P\to\{\ast\}$ induced by $\mu$. The metric $\tilde{\eta}$ obtained from \th\ref{T: existence of G->M invariant metrics on P->M} is a $G$-invariant Riemannian metric on $P$. 
\end{remark}


\section{Cheeger-like deformation for Lie groupoid actions}\label{S: Cheeger like deformation for Lie groupoids actions}

In this section we give the construction on how to deform an action Lie groupoid  onto the orbit space in an analogous fashion to the Cheeger deformation procedure for Lie group actions.

Consider an action $\mu$  of the proper Lie groupoid $\G\rightrightarrows M$ on $P$ along $\alpha\colon P\to M$, and the action groupoid $\G\times_M P\rightrightarrows P$. Consider $\eta^P$ a $\mu$-transversely invariant Riemannian metric on $P$ (see \th\ref{T: existence of G->M invariant metrics on P->M}), and $\eta^{(1)}$ a $1$-metric on $\G\rightrightarrows M$ induced by a $2$-metric. Set 
\begin{equation}\label{EQ: Cheeger deformation metric}
\widehat{\eta}_{\varepsilon} = \frac{1}{\varepsilon}\eta^{(1)} + \eta^P
\end{equation}
on $\G\times P$, and restrict it to $\G\times_M P$.

We claim that this metric induces a series of metrics $\eta_\varepsilon$ on $P$, which generalizes Cheeger deformations.

For this, we present without proof a series of Lemmas.

\begin{lemma}
Consider $\mu$ a Lie groupoid action of the proper Lie groupoid $\G\rightrightarrows M$ on $\alpha\colon P\to M$, and $\G\times_M P\rightrightarrows P$ the action groupoid. Then for $p\in P$
\[
\overline{t}^{-1}(p) = \{(g,\mu(i(g),p))\mid g\in t^{-1}(\alpha(p))\} = t^{-1}(\alpha(p))\times L_p,
\]
where $L_p\subset P$ is the orbit of the action groupoid $\G\times_M P\rightrightarrows P$ that contains $p$.
\end{lemma}

With the previous description of the fibers of the submersion $\overline{t}\colon \G\times_M P\rightrightarrows P$ we get the following description for the $\overline{t}$-vertical space $\kernel D\overline{t}$.

\begin{lemma}\th\label{L: Kernel of bar(t)}
Consider  $\G\times_M P\rightrightarrows P$ the action groupoid of a Lie groupoid  action of the proper Lie groupoid $\G\rightrightarrows M$ on $\alpha\colon P\to M$. Then
\[
\kernel D_{(g,p)}\overline{t} = \kernel D_gt \times T_{p}L_p,
\]
where $L_p\subset P$ is the orbit of the action groupoid $\G\times_M P\rightrightarrows P$ containing $p$.
\end{lemma}

From this description of $\kernel D \overline{t}$, we deduce that $\bar{t}$ is a Riemannian submersion with respect to $\widehat{\eta}_\varepsilon$.

\begin{thm}\th\label{T: contraction of groupoid induces cheeger deformation}
Consider $\G\times_M P\rightrightarrows P$ the action groupoid of a Lie groupoid  action $\mu$ of the proper Lie groupoid $\G\rightrightarrows M$ on $\alpha\colon P\to M$. Equip the metric $\widehat{\eta}_\varepsilon$ on $\G\times_M P$. Then $\overline{t}\colon (\G\times_M P,\widehat{\eta}_\varepsilon)\to P$ is a Riemannian submersion.
\end{thm}

\begin{definition}\th\label{D: Cheeger deformation}
Given $\G\rightrightarrows M$ a proper Lie groupoid acting on $\alpha\colon  P\to M$ via $\mu$, an $\eta^{(1)}$ a $1$-metric on $\G$ induced by some $2$-metric, and a $\mu$-transversely invariant metric $\eta^P$ on $P$, we define the \emph{Cheeger deformation} of $\eta^P$ with respect to $\eta^{(1)}$ as the Riemannian metric $\eta_\varepsilon$ on $P$ induced by the Riemannian submersion $(\G\times_{M} P,\widehat{\eta}_\varepsilon)\to P$. 
\end{definition}

\subsection{Horizontal distribution} Given the Riemannian submersion $\overline{t}\colon (\G\times_M P,\eta^{(1)}+\eta^P)\to P$, we consider the vertical distribution $\ver\colon \G\times_M P\to T ( \G\times_M P)$, given by $\ver(g,p) = \kernel D_{(g,p)}\overline{t}$.

From now for $p\in P$ and  $X\in \kernel D_{\1_{\alpha(p)}}s\subset T_{\1_{\alpha(p)}}\G$, we denote by $X^\ast(p)\in T_p P$ the vector given by 
\[
	X^\ast(p) = -D_{(\1_{\alpha(p)},p)}\,\overline{t}(X,0).
\]

We define the \emph{Orbit tensor} $\Sh(p)\colon \kernel D_{\1_{\alpha(p)}}s\to \kernel D_{\1_{\alpha(p)}}s$ by the following relationship:
\begin{linenomath}
\begin{align*}
\eta^{(1)}(\Sh(p)(x),y) =& \eta^P(D_{(\1_{\alpha(p)},p)}\overline{t}(X,0),D_{(\1_{\alpha(p)},p)}\overline{t}(Y,0))\\
 =& \eta^P (X^\ast (p),Y^\ast (p)).
\end{align*}
\end{linenomath}
	
We give an alternative description of $\kernel D\overline{t}$, the vertical space of $\overline{t}$, that is similar to the one presented in \cite[Lemma~3.2]{Mueter}.

\begin{lemma}\th\label{L: Vertical space of target map of action Lie groupoid}
Let the proper Lie groupoid $\G\rightrightarrows M$ act effectively and smoothly on $\alpha\colon P\to M$. Then we have 
\[
	\ver(\1_{\alpha(p)},p) = \{(x,X^\ast(p))\mid x\in \kernel D_{\1_{\alpha(p)}} s\}.
\]
\end{lemma}

For the Riemannian metric $\widehat{\eta}_\varepsilon$ on the action groupoid $\G \times_M P\rightrightarrows P$, we denote by $\hor(g,p) = (\kernel D_{(g,p)}\,\overline{t})^\perp$ the horizontal space of the target map $\overline{t}\colon \G \times_M P\to P$, and proof the following description. 

\begin{thm}\th\label{T: bar(t)-Horizontal distribution}
Consider a proper Lie groupoid $\G\rightrightarrows M$ a Lie groupoid acting on $\alpha\colon P\to M$. Consider the metric $\widehat{\eta}_\varepsilon = \frac{1}{\varepsilon}\eta^{(1)}+\eta^{(P)}$ on the action groupoid $\G\times_M P\rightrightarrows P$. Then The $\overline{t}$-horizontal space with respect to $\widehat{\eta}_\varepsilon$ is given by:
\begin{linenomath}
\begin{align*}
\hor_\varepsilon(\1_{\alpha(p)},p) = \Bigg\{\left(-\varepsilon\Sh(p)(x)+\frac{1}{\varepsilon}\zeta,X^\ast(p)+\xi\right)\,\Big|\, & x\in \kernel D_{\1_{\alpha(p)}} s,\\
&\zeta\in (\kernel D_{\1_{\alpha(p)}} s)^\perp,\\
 &\mbox{ and } \xi\in  T_p(L_p)^\perp\Bigg\}.
\end{align*}
\end{linenomath}
\end{thm}

\subsection{Cheeger tensor} We define the \emph{Cheeger tensor of $\eta^P$ along the action of $\G\rightrightarrows M$ on $\alpha\colon P\to M$} for the metric $\eta_\varepsilon$ induced on $P$ by  \th\ref{T: contraction of groupoid induces cheeger deformation} to be the map $\Ch^{-1}_\varepsilon(p)\colon T_p P\to T_pP$ defined as follows: We write $X\in T_pP$ as $X = X^\top+X^\perp$, where $X^\top \in T_pL_p$ and $X^\perp \in \nu_p(L_p)$. Thus, there exists an $x\in \kernel D_{\1_{\alpha(p)}}\, s$ such that $X^\ast(p) = X^\top$. With this decomposition $X = X^\ast(p)+X^\perp$ we set
\[
\Ch^{-1}_\varepsilon(p)(X^\ast(p)+X^\perp ) = (\Id+\varepsilon\Sh(p)(x))^\ast(p)+X^\perp
\]

\begin{lemma}\th\label{L: Cheeger tensor does not depend of choice of x such that xast is X}
Let $\G\times_M P\rightrightarrows P$ be the groupoid action of a Lie groupoid action of the proper Lie groupoid $\G\rightrightarrows M$ on $\alpha\colon P\to M$. For $X\in TP$, such that $X(p)\in T_p L_p$ for all $p\in P$, the Cheeger tensor $\Ch^{-1}_\varepsilon(p)$ does not depend on the choice of $x\in \kernel D_{\1_{\alpha(p)}}\, s$ such that $-D_{\1_{\alpha(p)}}\overline{t}(x,0) = X(p)$. 
\end{lemma}

We define the \emph{horizontal lift at $p\in P$ of tangent vectors of $P$} as the map $h(p)\colon T_p P \to T_{(\1_{\alpha(p)},p)}\G\times_M P$ given by
\begin{equation}
h(p)(X) = h(p)(X^\ast(p)+X^\perp) = (-\varepsilon\Sh(p)(x), X).\tag{\mbox{Horizontal}}\label{E: Horizontal lifts of TGXP->P at 1(P)}
\end{equation}
Here we are using the fact that for $X^\top\in T_p(L_p)$ there exists an $x\in \kernel D_{\1_{\alpha(p)}}\, s$ such that $X^\ast(p) = X^\top$. We observe that it follows from the proof of \th\ref{L: Cheeger tensor does not depend of choice of x such that xast is X}, that $\Sh(p)(x)$ does not depend of the choice of $x\in \kernel D_{\1_{\alpha(p)}} s$, such that $X^\ast(p) = X^\top$. In other words, the definition of $h$ does not depend on the choice of $x\in \kernel D_{\1_{\alpha(p)}} s$.

\begin{remark}
From the description of $\hor(\1_{\alpha(p)},p)$ in \th\ref{T: bar(t)-Horizontal distribution} we conclude that $h(p)(X)\in \hor(\1_{\alpha(p)},p)$.
\end{remark}

The following theorem states that the horizontal vector $h(p)$ defined by \eqref{E: Horizontal lifts of TGXP->P at 1(P)} is the horizontal lift of the Cheeger tensor $\Ch^{-1}_\varepsilon(p)$.

\begin{thm}\th\label{T: h is horizontal lift of cheeger}
Consider $\G\rightrightarrows M$ a proper Lie groupoid acting on $\alpha\colon P\to M$ via $\mu$, and $\eta^{(1)}$ a $1$-metric  on $\G$. Let $\eta^P$ an $\mu$-transversely invariant Riemannian metric on $P$, and assume that for all $p\in M$ we have $\nu_p (L_p)\subset T_p \alpha^{-1}(\alpha(p))$. We consider $-D\overline{t}\colon T(\G\times_M P)\to TP$. Then for all $X\in T_p P$, the vector $h(p)(X)$ is an $D\overline{t}$-horizontal lift of $Ch^{-1}_\varepsilon(p)(X)$, i.e.
\[
D\overline{t}(h(p)(X)) = Ch^{-1}_\varepsilon(p)(X).
\]
\end{thm}

We can describe the metric $\eta_\varepsilon$ on $P$ explicitly in terms of our original metric $\eta^P$ on $P$ and the Cheeger tensor.

\begin{thm}\th\label{T: cheeger deformation controlled by cheeger tensor}
Consider $\G\rightrightarrows M$ a proper Lie groupoid acting on $\alpha\colon P\to M$  via $\mu$, and $\eta^{(1)}$  a $1$-metric on  $\G$. Let $\eta^P$ be a $\mu$-transversely invariant Riemannian metric on $P$, and assume that for all $p\in M$ we have $\nu_p (L_p)\subset T_p \alpha^{-1}(\alpha(p))$. Then for $X,Y\in T_pP$ we have
\[
\eta_\varepsilon(X,Y) = \eta^P(Ch_\varepsilon(p)(X),Y).
\]
\end{thm}

With this we see that $\eta_\varepsilon$ is a deformation of our original $\mu$-transversely invariant metric $\eta^{P}$ on $P$.

\begin{thm}\th\label{T: Cheeger deformation collapses}
Consider $\G\rightrightarrows M$ a proper Lie groupoid acting on $\alpha\colon P\to M$ via $\mu$, and $\eta^{(1)}$ a $1$-metric on $\G$. Let $\eta^P$ be a $\mu$-transversely invariant Riemannian metric on $P$, and assume that for all $p\in M$ we have $\nu_p (L_p)\subset T_p \alpha^{-1}(\alpha(p))$. Then setting $\fol =\{L_p\mid p\in P\}$, we have that $(P,d_{\eta_{\varepsilon}})\to (P/\fol, d_{\eta^P}^\ast)$ in the Gromov-Hausdorff topology as $\varepsilon\to \infty$, and $\eta_{\varepsilon}\to \eta^P$ in the $C^{\infty}$-topology as $\varepsilon\to 0$. 
\end{thm}

\subsection{Curvature} In this section we compute the sectional curvature of the deformed Riemannian manifold $(P,\eta_\varepsilon)$. We can use  O'Neil's formula (\th\ref{T: ONeills formula}) to compute the sectional curvature of our new deformed metrics.\\

\begin{duplicate}[\ref{MT: Cheeger deformation collapses}]
Consider $\G\rightrightarrows M$ a proper Lie groupoid acting on $\alpha\colon P\to M$ a submersion, and $\eta^{(1)}$ a Riemannian metric on $\G$ making it a Riemannian Lie groupoid. Let $\eta^P$ a left-$\G$-invariant Riemannian metric on $P$, and assume that for all $p\in M$ we have $\nu_p(L_p)\subset T_p \alpha^{-1}(\alpha(p))$. Then for the Cheeger deformation $(P,\eta_\varepsilon)$ we have for $v,w\in T_pP$ linearly independent vectors, with $v = X^\ast(p)+v^\perp$ and $w = Y^\ast(p)+w^\perp$ for some $x,y\in \kernel D_{\1_{\alpha(p)}} s$. Then
\begin{linenomath}
\begin{equation}\tag{\ref{EQ: Full curvature description}}
\begin{split}
K_{\eta_\varepsilon}\Big(\Ch^{-1}_\varepsilon(p)&(v),\Ch^{-1}_\varepsilon(p)(w)\Big)  = K_{\eta^P}(v,w)+\varepsilon^3K_{\eta^{(1)}}\Big(\Sh(p)(v),\Sh(p)(w)\Big)\\
+&3\Big\|A_{h(p)(v)}h(p)(w)\Big\|^2_{\left(\frac{1}{\varepsilon}\eta^{(1)}+\eta^P\right)}\\
+&\Big\|\II_\varepsilon\big((v,-\varepsilon \Sh(p)(v)),(w,-\varepsilon\Sh(p)(w))\big)\Big\|^2_{\frac{1}{\varepsilon}\eta^{(1)}+\eta^P}\\
 -&\left(\frac{1}{\varepsilon}\eta^{(1)}+\eta^P\right)\Big(\II_\varepsilon\big((v,-\varepsilon\Sh(p)(v)),(v,-\varepsilon\Sh(p)(v))\big),\\
 &\II_\varepsilon\big((w,-\varepsilon\Sh(p)(w)),(w,-\varepsilon\Sh(p)(w))\big)\Big).
\end{split}
\end{equation}
\end{linenomath}
$\II_\varepsilon(\cdot,\cdot)$ is the second fundamental form of $(\G\times_M P,\widehat{\eta}_\varepsilon)\subset (\G\times P,(1/\varepsilon)\eta^{(1)}+\eta^{P})$, and $A_\cdot \cdot$ is the $A$-tensor of the Riemannian submersion $\bar{t}\colon (\G\times_M P,\hat{\eta}_\varepsilon)\to (P,\eta_t)$.
\end{duplicate}

\begin{remark}
We point out that for a Lie group action $\mu$ of  $G\rightrightarrows\{\ast\}$ along $\alpha\colon M\to\{\ast\}$, we automatically have that $\nu_p L_p \subset T_p M = T_p \alpha^{-1}(\ast) = T_p \alpha^{-1}(\alpha(p))$. 
\end{remark}

\begin{remark}
In the case when we consider a group action $\mu$ of  $G\rightrightarrows\{\ast\}$ along $\alpha\colon M\to\{\ast\}$, the action groupoid is $G\times M\rightrightarrows M$. In this case the second fundamental form term in \eqref{EQ: Full curvature description} does not appear. Thus we recover the expression \eqref{EQ: Cheeger deformation curvature}.
\end{remark}


\section{Connection between Singular Riemannian foliations and Lie groupoid actions}\label{S: Connection between Singular Riemannian foliations and Lie groupoid actions}

Here we present how closed singular Riemannian foliations connect with Lie groupoid actions.

Let $(M,\fol,\g)$ be a closed singular Riemannian foliation, and fix $p\in M$. We denote by $\Omega_p^\ast$ the set of all piecewise smooth loops based at $p$. By \cite{MendesRadeschi2015}, for each $\gamma\in \Omega_p^\ast$ there exists a foliated isometry $\Gamma_\gamma\colon (\nu_p(L_p),\fol_p)\to (\nu_p(L_p),\fol_p)$. We point out that the group 
\[
\Hol(L,p):= \{\Gamma_\gamma\mid\gamma\in \Omega_p^\ast\}
\] 
is a Lie group: Consider $\Hol(L,p)_0$ the connected component of $\Hol(L,p)$ containing the identity the map; it is a path connected subgroup of the Lie group $\mathrm{O}(\Sp^\perp_p)$ and thus a Lie group by \cite{Yamabe1950}. The map 
\[
\overline{\Hol}\colon \pi_1(L,p)\to \Hol(L,p)/\Hol(L,p)_0\subset \mathrm{O}(\Sp^\perp_{p},\fol_{p})/\mathrm{O}(\fol_p)
\] 
is a  well defined surjective group morphism, \cite[Proposition 2.5]{Corro2019}. In particular 
\[
|\pi_0(\Hol(L,p))|= |\Hol(L,p)/\Hol(L,p)_0|\leq |\pi_1(L,p)|.
\] 
But since $L$ is a smooth manifold, $\pi_1(L,p)$ is a countable set \cite[Prop. 1.16]{Lee}. This implies that $\Hol(L,p)$ is a group whose number of connected components are countable, and as mentioned before each connected component is a Lie group. Thus, by definition $\Hol(L,p)$ is a smooth manifold. From this it follows that $\Hol(L,p)$ is a Lie group.

For a singular Riemannian foliation $(M,\mathcal{F})$ with connected leaves, the holonomy  group of a leaf $L$ depends on the base point $p$ up to conjugation: let $q\in L$ be in the same connected component of $L$ as $p$. Then we have that
\[
\Hol(L,q)= G^{-1}\Hol(L,p)G
\]
for some foliated isometry $G\colon (\Sp_q^\perp,\fol_q)\to (\Sp^\perp_p,\fol_p)$ determined by a path in $L$ joining $p$ to $q$.  

Now since for a closed foliation $\fol$ the leaves are embedded submanifolds, for a fixed leaf $L\in \fol$ we can identify the tubular neighborhood $B_r(L)$ with the normal bundle $\nu(L)\to L$, and there exists an Ehresmann connection $H$, which induces a connection $\nabla^H$  on $\nu(L)$ such that $\Hol(\nabla^H,p)= \Hol(L,p)$ (see \cite{Alexandrino2010,AlexandrinoInagakiStruchiner2018}). 

We denote the isometries of $(\Sp^\perp_p,\fol_p)$ that map leaves to leaves by $\Orth(\Sp^\perp,\fol_p): = \{h\in \mathrm{Iso}(\Sp^\perp_p)\mid h(L_v)\subset L_{h(v)}\}$, and we set $\Orth(\fol_p):= \{h\in \Orth(\Sp^\perp,\fol_p)\mid h(v)\in L_v\}$ (these are Lie groups by \cite{CorroGalazGarcia2024} ). We observe that when $\dim(\Sp^\perp_p) = k_p$, then $\Iso(\Sp^\perp_p)=\Orth(k_p)$.

Let $K_p\subset \Orth(\fol_p)$ be the maximal connected Lie subgroup in $\Orth(\fol_p)$ and denote by $\Hol^H\rightrightarrows \nu(L)$ the Lie groupoid generated by all parallel translations with respect to $\nabla^H$ along piecewise smooth curves $\gamma\colon [0,1]\to L$. We define \emph{the linearized foliation $\fol^\ell_L$ of $\mathrm{fol}$ around $L$} as the foliation
\[
\fol^\ell_L := \Big\{P_\alpha(gv)\mid P_\alpha\mbox{ parallel transport along }\alpha\colon I\to L,\: g\in K_p,\: v\in \D^\perp_p\Big\}.
\]
Via the normal exponential map the foliation $\fol^\ell_L$ on $\nu(L)$ induces a foliation, also denoted $\fol^\ell_L$, on $B_r(L)$. Then we have for a fixed $p\in B_r(L)$ that $L^\ell_p\subset L_p$, i.e. $\fol^\ell_L\subset \fol|_{B_r(L)}$. Given a closed singular Riemannian foliation $\fol$ on a compact manifold $M$ if for any $L\in \fol$ we have $\fol^\ell_L= \fol$ in some small tubular neighborhood of $L$ we say that $\fol$ is an \emph{orbit-like foliation}.  This is equivalent to the infinitesimal foliation $\fol_p$ being homogeneous at any point (see \cite{AlexandrinoRadeschi2016}). 

With these concepts defined we can present the following theorem that gives the local connection between singular Riemannian foliations and Lie groupoid actions.

\begin{thm}[Theorem 1.5 in \cite{AlexandrinoInagakiStruchiner2018}]\th\label{T: Local transversal homogeneous subfoliation}
Let $(M,\fol,\g)$ be a closed singular Riemannian foliation on a complete Riemannian manifold. Then for $L\in \fol$ there exists a saturated small tubular neighborhood $B_r(L)$, a Lie groupoid $\mathcal{G}\rightrightarrows L$ and a left representation of this Lie groupoid on $\nu(L)\to L$ such that by identifying a small saturated tubular neighborhood of the $0$-section on $\nu(L)$ with $B_r(L)$, the orbits of the groupoid $\mathcal{G}\times_L \nu(L)\rightrightarrows \nu(L)$ on this tubular neighborhood can be identified with the leaves of $\fol^\ell_L$ on $B_r(L)$.
\end{thm}

Informally, the groupoid $\G\rightrightarrows M$ is determined by the holonomy of the leaf.

In particular when the foliation $\fol$ is and orbit-like foliation, the foliation is locally given by Lie groupoid actions.

\begin{cor}\label{C: orbit-like foliation locally given by a Lie groupoid action}
Let $(M,\fol,\g)$ be an orbit-like foliation. The for any leaf $L\in \fol$, there exists a saturated tubular neighborhood $U$ of $L$ such that $(U,\fol|_{U})$ is foliated diffeomorphic to the foliation induced by a Lie groupoid left representation $\mu$ of a Lie group $\G_L\rightrightarrows L$ on the normal bundle $\nu(L)\to L$.
\end{cor}

\subsection{Example: Codimension 1 closed singular Riemannian foliations on compact manifolds} Here we present a particular example of a family of orbit-like foliations.  We begin by recalling the Slice Theorem for closed singular Riemannian foliations.

\begin{thm}[Slice Theorem, \cite{MendesRadeschi2015}]\th\label{T: Slice Theorem}
Let $(M,\fol)$ be a closed singular Riemannian foliation and $L\in \fol$. Then, there exists a sufficiently small $r>0$ such that a distance tubular neighborhood $B_r(L)$ is saturated and it is foliated diffeomorphic to 
\[
	(P\times_{\Hol(L,p)} \D^\perp_p,P\times_{\Hol(L,p)} \fol_p),
\]
where $P$ is a $\Hol(L,p)$-principal bundle over $L$. For $[\xi,v]\in P\times_{\Hol(L,p)} \D^\perp_p$ we have $[\zeta,w]\in L_{[\xi,v]}$ if and only if there exists $h\in K$ such that $h(v)=w$.
\end{thm}

We now consider $(M,\g)$ a compact Riemannian manifold equipped with a codimension one singular Riemannian foliation $\fol$ and consider the quotient map $\pi\colon M\to M/\fol$. The leaf space $M/\fol$ is homeomorphic to the closed interval $[-1,1]$, or to $\Sp^1$. The second case implies that $\pi$ is a fiber bundle, and the fibers of $\pi$ are the leaves of the foliation.

Moreover, given $p_{\pm}\in L_{\pm}$ by \cite[Proposition 1.4]{CorroMoreno2020}, \cite[Proposition p. 184]{Moreno2019} we have that the spaces of directions of $M/\fol$ at $\pm 1$ is homeomorphic to 
\[
(\Sp^{\perp}_{p_{\pm}}/\fol_{p_{\pm}})/\overline{\Hol}(L_{\pm},p_{\pm}).
\]
Since $M/\fol$ is homeomorphic to $[-1,1]$, and $\pi(L^{\pm})=\pm 1$, then we have that 
\[
(\Sp^{\perp}_{p_{\pm}}/\fol_{p_{\pm}})/\overline{\Hol}(L_{\pm},p_{\pm}) \cong \ast_{\pm}.
\]
Since $\overline{\Hol}(L_{\pm},p_{\pm})$ is discrete, we conclude that $\Sp^{\perp}_{p_{\pm}}/\fol_{p_{\pm}}$ is homeomorphic to a discrete space. 

We observe that either $\dim(\Sp^{\perp}_{p_{\pm}})=0$ or $\dim(\Sp_{p_{\pm}^{\perp}})\geq 1$. In the first case when $\dim(\Sp^{\perp}_{p_{\pm}})=0$, we observe that $\fol_{p_{\pm}}$ is either the trivial points-leaves foliation or the trivial one-leaf foliation. The case when $\fol_{p_{\pm}}$ is the trivial points-leaves foliation then $\Sp_{\pm}^\perp/\fol_{p_{\pm}}\cong \Sp^0$ and $\overline{\Hol}(L_{\pm},p_{\pm})\cong \Z_2$. In the case when $\fol_{p_{\pm}}$ is the trivial single-leaf foliation, then we have that $\Sp^{\perp}_{p_{\pm}}/\fol_{p_{\pm}}$ is a single point and thus homeomorphic to a point and $\overline{\Hol}(L_{\pm},p_{\pm})$ is trivial, i.e. $\Hol(L_{\pm},p_{\pm})$ is connected. The second case when $\dim(\Sp_{p_{\pm}^{\perp}})\geq 1$, implies that the foliation $\fol_{p_{\pm}}$ is a trivial single-leaf foliation since $\Sp_{p_{\pm}}^{\perp}/\fol_{p_{\pm}}$ is connected and $\overline{\Hol}(L_\pm,p_\pm)$ is discrete. 

Summarizing, by \th\ref{T: Slice Theorem}, each tubular neighborhood $B_r(L_{\pm})$ is foliated diffeomorphic to a disk bundle 
\[
\overline{P}_\pm:=P_{\pm}\times_{\Hol(L_{\pm},p_{\pm})} \D^{\perp}_{p_{\pm}}\to L_{\pm},
\]
and the foliation $\fol_{\pm}|_{B_{r}(L_{\pm})}$ is the cone of the trivial foliation $P_{\pm}\times_{\Hol(L_{\pm},p_{\pm})}\Sp^{\perp}_{p_{\pm}}$.

Thus we conclude that $M$ is foliated diffeomorphic to a double disk bundle decomposition 
\[
\overline{P}_{-}\bigsqcup_{h} \overline{P}_{+},
\]
where $h$ is a homeomorphism $h\colon\partial\overline{P}_{-}\cong (P_{-}\times_{\Hol(L_{-},p_{-})}\Sp_{p_{-}}^{\perp})\to (P_{+}\times_{\Hol(L_{+},p_{+})}\Sp_{p_{+}}^{\perp}\cong \overline{P}_{+})$.

\begin{lemma}\th\label{L: Codim 1 SRF are orbit like}
Let $\fol$ be a closed codimension $1$ singular Riemannian foliation on a compact Riemannian manifold $(M,\g)$. Then $\fol$ is an orbit-like foliation.
\end{lemma}

\begin{proof}
As we have seen above, for any point $p\in M$, the infinitesimal foliation of  $L_p$ is the trivial leaf foliation $\{\Sp^{\perp}_{p}\}$, if $L_p$ is one of the singular leaves and has codimension at least $2$, or it is trivial point-foliation $\{\{-\},\{+\}\}$ on $\Sp^{\perp}_{p}\cong \Sp^{0}\cong \{-,+\}$. In both cases we have that $\Orth(\fol_{p}) \cong \Orth(\dim(L_{p})+1)$, and thus the maximal connected Lie subgroup of $\Orth(\fol_{p})$ is $\SO(\dim(L_{p})+1)$. The conclusion follows from the observation that $\SO(\dim(L_{p})+1)$ acts transitively on $\Sp_{p}^{\perp}$, which is the only leaf of the infinitesimal foliation $\fol_p$.
\end{proof}

\begin{remark}
We point out that the Lie groupoids $\G_\pm\rightrightarrows L_{\pm}$ depend heavily on the topology of of the singular leaves, and thus the Lie groupoids $\G_\pm\times_{L_\pm}\nu(L_{\pm})\rightrightarrows \nu(L_{\pm})$ in \th\ref{T: Local transversal homogeneous subfoliation} which determine locally the foliation $\fol$, even in the case of a Lie group action. This can lead to  possibly having different ``rates'' of collapse. This leads to problems when we try to re-glue the deformed tubular neighborhoods of the singular leaves $L_\pm$.
\end{remark}

\begin{example}
We consider the action of $\Sp^1$ on the round $\R P^2$ of cohomogeneity one by isometries. The orbit space with the induced metric is isometric to $[0,\pi/2]$. There are two singular orbits: a fixed point $L_{\pi/2}$, corresponding to the preimage of $\pi/2$, and a circle orbit $L_0$ with isotropy $\Z_2$ corresponding to $0$. The foliation around $L_0$ is the circle foliation on the Möbius band, and the foliation around the fixed point is diffeomorphic to a circle action on $\D^2\subset \R^2$ by rotations around the origin. In this case the rate of convergence is controlled globally by the circle action, so  in this case the deformed metrics on each piece agree with one another around the boundaries of the pieces.
\end{example}

Thus it is natural to ask the following problem.

\begin{problem}
Let $(M,\fol,\g)$ be a closed singular Riemannian foliation such that $M/\fol$ is homeomorphic to $[-1,1]$. Let $L_\pm$ be the leaves corresponding to $\pm 1\in [-1,1]$. Consider the the Lie groupoid representations $(\G_\pm \times_{L_{\pm}} \nu(L_{\pm}),\g^{\pm}_\varepsilon)$ equipped with Cheeger deformation metrics give by Section~\ref{S: Cheeger like deformation for Lie groupoids actions}. How does the deformation $\g^-_\varepsilon$ relate to  $\g^+_\varepsilon$ on a tubular neighborhood of a regular leaf?
\end{problem}

Shedding light on this problem would allow us to determine if the foliated Cheeger deformation metrics $\g^{\pm}_\varepsilon$ can be glued together in a way that we preserve control of the curvature to obtain a global deformation.

\begin{remark}
One can always use a foliated partition of unity \cite{CorroFernandezPerales22} to glue the metrics $\g^{\pm}_{t}$ to get a global deformation, but with the cost of losing control on the sectional curvature on the overlap of the two tubular neighborhoods around the leaves $L_\pm$.
\end{remark}

Even in the case when the boundaries $(\partial\overline{P}_\pm,\g^\pm_\varepsilon|_{\partial\overline{P}_\pm})$ are isometric, we need some more information about the geometry around them to obtain a a smooth Riemannian metric by gluing their boundaries via an isomorphism. For example from \cite[Corollary B]{ReiserWraith2023} it follows that when the sum of the second fundamental forms of $\partial\overline{P}_\pm\hookrightarrow (\overline{P}_\pm,\g^\pm_\varepsilon)$ is positive definite, then we can glue $(\overline{P}_-,\g^-_\varepsilon)$ to $(\overline{P}_+,\g^+_\varepsilon)$ to obtain a smooth metric $\g_\varepsilon$ Ricci curvature bounded below by some constant $\kappa_\varepsilon$ that depends on $\varepsilon$. From \th\ref{MT: Cheeger deformation collapses} we see that this constant $\kappa_\varepsilon$ depends on the second fundamental form $\II_\varepsilon$.

Thus for some cases we might share some light on the following problem, related to \th\ref{ConjectureGrove}.

\begin{problem}
Given $(M,\fol,\g)$ a singular Riemannian foliation of codimension $1$ with closed leaves on a compact simply-connected manifold, does there exists a new (maybe foliated) Riemannian metric of positive or non-negative Ricci curvature.
\end{problem}

This problem has a positive answer in the homogeneous case \cite{GroveZiller2002}.


%

\bibliographystyle{siam}
\bibliography{Bibliography}

\begin{thebibliography}{10}

\bibitem{Alexandrino2010}
{\sc M.~M. Alexandrino}, {\em Desingularization of singular {R}iemannian
  foliation}, Geom. Dedicata, 149 (2010), pp.~397--416.

\bibitem{AlexandrinoBettiol}
{\sc M.~M. Alexandrino and R.~G. Bettiol}, {\em Lie groups and geometric
  aspects of isometric actions}, Springer, Cham, 2015.

\bibitem{AlexandrinoBriquetToeben2012}
{\sc M.~M. Alexandrino, R.~Briquet, and D.~T\"oben}, {\em Progress in the
  theory of singular {R}iemannian foliations}, Differential Geom. Appl., 31
  (2013), pp.~248--267.

\bibitem{AlexandrinoCavenaghiCorroInagaki2024}
{\sc M.~M. Alexandrino, L.~F. Cavenaghi, D.~Corro, and M.~K. Inagaki}, {\em
  Singular {R}iemannian foliations, variational problems and {P}rinciples of
  {S}ymmetric {C}riticalities},
  \href{https://arxiv.org/pdf/2311.07058.pdf}{arXiv:2311.07058 [math.DG]},
  (2024).

\bibitem{AlexandrinoInagakiStruchiner2018}
{\sc M.~M. Alexandrino, M.~K. Inagaki, M.~de~Melo, and I.~Struchiner}, {\em Lie
  groupoids and semi-local models of singular {R}iemannian foliations}, Ann.
  Global Anal. Geom., 61 (2022), pp.~593--619.

\bibitem{AlexandrinoRadeschi2016}
{\sc M.~M. Alexandrino and M.~Radeschi}, {\em Mean curvature flow of singular
  {R}iemannian foliations}, J. Geom. Anal., 26 (2016), pp.~2204--2220.

\bibitem{AlexandrinoRadeschi2017}
\leavevmode\vrule height 2pt depth -1.6pt width 23pt, {\em Closure of singular
  foliations: the proof of {M}olino's conjecture}, Compos. Math., 153 (2017),
  pp.~2577--2590.

\bibitem{Bieberbach}
{\sc L.~Bieberbach}, {\em \"{U}ber die {B}ewegungsgruppen der {E}uklidischen
  {R}\"{a}ume ({Z}weite {A}bhandlung.) {D}ie {G}ruppen mit einem endlichen
  {F}undamentalbereich}, Math. Ann., 72 (1912), pp.~400--412.

\bibitem{Brandt1927}
{\sc H.~Brandt}, {\em \"uber eine {V}erallgemeinerung des {G}ruppenbegriffes},
  Math. Ann., 96 (1927), pp.~360--366.

\bibitem{BuragoBuragoIvanov}
{\sc D.~Burago, Y.~Burago, and S.~Ivanov}, {\em A course in metric geometry},
  vol.~33 of Graduate Studies in Mathematics, American Mathematical Society,
  Providence, RI, 2001.

\bibitem{CavenaghiESilvaSperanca2023}
{\sc L.~Cavenaghi, R.~e~Silva, and L.~Speran{\c{c}}a}, {\em Positive ricci
  curvature through cheeger deformations}, Collect. Math,  (2023).

\bibitem{Cheeger1973}
{\sc J.~Cheeger}, {\em Some examples of manifolds of nonnegative curvature}, J.
  Differential Geometry, 8 (1973), pp.~623--628.

\bibitem{CheegerFukayaGromov1992}
{\sc J.~Cheeger, K.~Fukaya, and M.~Gromov}, {\em Nilpotent structures and
  invariant metrics on collapsed manifolds}, J. Amer. Math. Soc., 5 (1992),
  pp.~327--372.

\bibitem{CheegerGromov1986}
{\sc J.~Cheeger and M.~Gromov}, {\em Collapsing {R}iemannian manifolds while
  keeping their curvature bounded. {I}}, J. Differential Geom., 23 (1986),
  pp.~309--346.

\bibitem{CheegerGromov1990}
\leavevmode\vrule height 2pt depth -1.6pt width 23pt, {\em Collapsing
  {R}iemannian manifolds while keeping their curvature bounded. {II}}, J.
  Differential Geom., 32 (1990), pp.~269--298.

\bibitem{Corro}
{\sc D.~Corro}, {\em Manifolds with aspherical singular {R}iemannian
  foliations}, PhD thesis, Karlsruhe Institute of Technology, 2018.
\newblock \url{https://publikationen.bibliothek.kit.edu/1000085363}.

\bibitem{Corro2019}
\leavevmode\vrule height 2pt depth -1.6pt width 23pt, {\em A-foliations of
  codimension two on compact simply-connected manifolds}, Math. Z., 304 (2023),
  pp.~Paper No. 63, 32.

\bibitem{Corro2024}
\leavevmode\vrule height 2pt depth -1.6pt width 23pt, {\em Collapsing regular
  {R}iemannian foliations with flat leaves},
  \href{https://arxiv.org/pdf/2403.11602}{arXiv:2403.11602 [math.DG]},  (2024).

\bibitem{Corro2025}
\leavevmode\vrule height 2pt depth -1.6pt width 23pt, {\em Cheeger
  {D}eformations for {L}ie groupoid actions},
  \href{https://arxiv.org/pdf/2502.01460}{arXiv:2502.01460 [math.DG]},  (2025).

\bibitem{CorroFernandezPerales22}
{\sc D.~Corro, J.~C. Fernandez, and R.~Perales}, {\em Yamabe problem in the
  presence of singular {R}iemannian foliations}, Calc. Var. Partial
  Differential Equations, 62 (2023), pp.~Paper No. 26, 34.

\bibitem{CorroGalazGarcia2016}
{\sc D.~Corro and F.~Galaz-Garc\'ia}, {\em Positive {R}icci curvature on
  simply-connected manifolds with cohomogeneity-two torus actions}, Proc. Amer.
  Math. Soc., 148 (2020), pp.~3087--3097.

\bibitem{CorroGalazGarcia2024}
{\sc D.~Corro and F.~Galaz-Garc\'{i}a}, {\em Myers-{S}teenrod theorems for
  metric and singular {R}iemannian foliations},
  \href{https://arxiv.org/abs/2407.03534}{arXiv:2307.03534},  (2023).

\bibitem{CorroMoreno2020}
{\sc D.~Corro and A.~Moreno}, {\em Core reduction for singular {R}iemannian
  foliations and applications to positive curvature}, Ann. Global Anal. Geom.,
  62 (2022), pp.~617--634.

\bibitem{CrainicMestreStruchiner2020}
{\sc M.~Crainic, J.~a.~N. Mestre, and I.~Struchiner}, {\em Deformations of
  {L}ie groupoids}, Int. Math. Res. Not. IMRN,  (2020), pp.~7662--7746.

\bibitem{Debord2001}
{\sc C.~Debord}, {\em Holonomy groupoids of singular foliations}, J.
  Differential Geom., 58 (2001), pp.~467--500.

\bibitem{delHoyoFernandes2018}
{\sc M.~del Hoyo and R.~L. Fernandes}, {\em Riemannian metrics on {L}ie
  groupoids}, J. Reine Angew. Math., 735 (2018), pp.~143--173.

\bibitem{delHoyo2013}
{\sc M.~L. del Hoyo}, {\em Lie groupoids and their orbispaces}, Portugal.
  Math., 70 (2013), pp.~161--209.

\bibitem{EhresmannCompltes}
{\sc C.~Ehresmann}, {\em Oeuvres compl\`etes et comment\'ees.}, 1980--1984.

\bibitem{EpsteinRosenberg1977}
{\sc D.~B.~A. Epstein and H.~Rosenberg}, {\em Stability of compact foliations},
  in Geometry and topology ({P}roc. {III} {L}atin {A}mer. {S}chool of {M}ath.,
  {I}nst. {M}at. {P}ura {A}plicada {CNP}q, {R}io de {J}aneiro, 1976), vol.~Vol.
  597 of Lecture Notes in Math., Springer, Berlin-New York, 1977, pp.~151--160.

\bibitem{FangRong2005}
{\sc F.~Fang and X.~Rong}, {\em Homeomorphism classification of positively
  curved manifolds with almost maximal symmetry rank}, Math. Ann., 332 (2005),
  pp.~81--101.

\bibitem{FarrellJones1998}
{\sc F.~T. Farrell and L.~E. Jones}, {\em Collapsing foliated {R}iemannian
  manifolds}, Asian J. Math., 2 (1998), pp.~443--494.

\bibitem{FarrellWu2018}
{\sc F.~T. Farrell and X.~Wu}, {\em Riemannian foliation with exotic tori as
  leaves}, Bull. Lond. Math. Soc., 51 (2019), pp.~745--750.

\bibitem{FerusKarcherMuenzner1981}
{\sc D.~Ferus, H.~Karcher, and H.~F. M\"{u}nzner}, {\em Cliffordalgebren und
  neue isoparametrische {H}yperfl\"{a}chen}, Math. Z., 177 (1981),
  pp.~479--502.

\bibitem{Fukaya1987}
{\sc K.~Fukaya}, {\em Collapsing {R}iemannian manifolds to ones of lower
  dimensions}, J. Differential Geom., 25 (1987), pp.~139--156.

\bibitem{Fukaya1989}
\leavevmode\vrule height 2pt depth -1.6pt width 23pt, {\em Collapsing
  {R}iemannian manifolds to ones with lower dimension. {II}}, J. Math. Soc.
  Japan, 41 (1989), pp.~333--356.

\bibitem{GalazGarciaRadeschi2015}
{\sc F.~Galaz-Garcia and M.~Radeschi}, {\em Singular {R}iemannian foliations
  and applications to positive and non-negative curvature}, J. Topol., 8
  (2015), pp.~603--620.

\bibitem{GallegoGualdraniHectorReventos1989}
{\sc E.~Gallego, L.~Gualandri, G.~Hector, and A.~Revent\'{o}s}, {\em
  Groupo\"{\i}des riemanniens}, Publ. Mat., 33 (1989), pp.~417--422.

\bibitem{GeRadeschi2013}
{\sc J.~Ge and M.~Radeschi}, {\em Differentiable classification of 4-manifolds
  with singular {R}iemannian foliations}, Math. Ann., 363 (2015), pp.~525--548.

\bibitem{GilkeyParkTuschmann1998}
{\sc P.~B. Gilkey, J.~Park, and W.~Tuschmann}, {\em Invariant metrics of
  positive {R}icci curvature on principal bundles}, Math. Z., 227 (1998),
  pp.~455--463.

\bibitem{Glickenstein2008}
{\sc D.~Glickenstein}, {\em Riemannian groupoids and solitons for
  three-dimensional homogeneous {R}icci and cross-curvature flows}, Int. Math.
  Res. Not. IMRN,  (2008), pp.~Art. ID rnn034, 49.

\bibitem{GonzalezAlvaroGuijarro2023}
{\sc D.~Gonz\'alez-\'Alvaro and L.~Guijarro}, {\em On the double soul
  conjecture}, in New trends in geometric analysis---{S}panish {N}etwork of
  {G}eometric {A}nalysis 2007--2021, vol.~10 of RSME Springer Ser., Springer,
  Cham, 2023, pp.~227--244.

\bibitem{GromollWalschap}
{\sc D.~Gromoll and G.~Walschap}, {\em Metric foliations and curvature},
  vol.~268 of Progress in Mathematics, Birkh\"auser Verlag, Basel, 2009.

\bibitem{Grove2002}
{\sc K.~Grove}, {\em Geometry of, and via, symmetries}, in Conformal,
  {R}iemannian and {L}agrangian geometry ({K}noxville, {TN}, 2000), vol.~27 of
  Univ. Lecture Ser., Amer. Math. Soc., Providence, RI, 2002, pp.~31--53.

\bibitem{GroveKarcher1973}
{\sc K.~Grove and H.~Karcher}, {\em How to conjugate {$C^{1}$}-close group
  actions}, Math. Z., 132 (1973), pp.~11--20.

\bibitem{GroveSearle1994}
{\sc K.~Grove and C.~Searle}, {\em Positively curved manifolds with maximal
  symmetry-rank}, J. Pure Appl. Algebra, 91 (1994), pp.~137--142.

\bibitem{GroveZiller2000}
{\sc K.~Grove and W.~Ziller}, {\em Curvature and symmetry of {M}ilnor spheres},
  Ann. of Math. (2), 152 (2000), pp.~331--367.

\bibitem{GroveZiller2002}
\leavevmode\vrule height 2pt depth -1.6pt width 23pt, {\em Cohomogeneity one
  manifolds with positive {R}icci curvature}, Invent. Math., 149 (2002),
  pp.~619--646.

\bibitem{Hamilton1978}
{\sc R.~S. Hamilton}, {\em Deformation theory of foliations}, 1978.
\newblock preprint, Cornell University.

\bibitem{LawsonYau1972}
{\sc H.~B. Lawson, Jr. and S.~T. Yau}, {\em Compact manifolds of nonpositive
  curvature}, J. Differential Geometry, 7 (1972), pp.~211--228.

\bibitem{Lee}
{\sc J.~M. Lee}, {\em Introduction to smooth manifolds}, vol.~218 of Graduate
  Texts in Mathematics, Springer, New York, second~ed., 2013.

\bibitem{Lytchak2010}
{\sc A.~Lytchak}, {\em Geometric resolution of singular {R}iemannian
  foliations}, Geom. Dedicata, 149 (2010), pp.~379--395.

\bibitem{LytchakThobergsson2010}
{\sc A.~Lytchak and G.~Thorbergsson}, {\em Curvature explosion in quotients and
  applications}, J. Differential Geom., 85 (2010), pp.~117--139.

\bibitem{MendesRadeschi2015}
{\sc R.~A.~E. Mendes and M.~Radeschi}, {\em A slice theorem for singular
  {R}iemannian foliations, with applications}, Trans. Amer. Math. Soc., 371
  (2019), pp.~4931--4949.

\bibitem{Moerdijk}
{\sc I.~Moerdijk and J.~Mr\v~cun}, {\em Introduction to foliations and {L}ie
  groupoids}, vol.~91 of Cambridge Studies in Advanced Mathematics, Cambridge
  University Press, Cambridge, 2003.

\bibitem{Molino}
{\sc P.~Molino}, {\em Riemannian foliations}, vol.~73 of Progress in
  Mathematics, Birkh\"auser Boston, Inc., Boston, MA, 1988.
\newblock Translated from the French by Grant Cairns, With appendices by
  Cairns, Y. Carri\`ere, \'E. Ghys, E. Salem and V. Sergiescu.

\bibitem{Moreno2019}
{\sc A.~Moreno}, {\em Point leaf maximal singular {R}iemannian foliations in
  positive curvature}, Differential Geom. Appl., 66 (2019), pp.~181--195.

\bibitem{MoullieWebpage}
{\sc L.~Moulli\'{e}}, {\em Cajun geometer-cheeger deformations}, 2017.
\newblock Last accessed 14 June 2023
  \url{https://lawrencemouille.wordpress.com/2017/02/27/cheeger-deformations/}.

\bibitem{Mueter}
{\sc M.~M\"{u}ter}, {\em Kr\"{u}mmungseh\"{o}hende {D}eformationen mittels
  {G}ruppenaktionen}, PhD thesis, Westf\"{a}lischen
  {W}ilhelms-{U}niversit\"{a}t {M}\"{u}nster, 1987.

\bibitem{ONeill1966}
{\sc B.~O'Neill}, {\em The fundamental equations of a submersion}, Michigan
  Math. J., 13 (1966), pp.~459--469.

\bibitem{RadeschiThesis}
{\sc M.~Radeschi}, {\em Low dimensional {S}ingular {R}iemannian {F}oliations in
  spheres}, PhD thesis, University of {P}ennsylvania, 2012.

\bibitem{Radeschi2014}
\leavevmode\vrule height 2pt depth -1.6pt width 23pt, {\em Clifford algebras
  and new singular {R}iemannian foliations in spheres}, Geom. Funct. Anal., 24
  (2014), pp.~1660--1682.

\bibitem{Reinhart1959}
{\sc B.~L. Reinhart}, {\em Foliated manifolds with bundle-like metrics}, Ann.
  of Math. (2), 69 (1959), pp.~119--132.

\bibitem{ReiserWraith2023}
{\sc P.~Reiser and D.~J. Wraith}, {\em A generalization of the perelman gluing
  theorem and applications},
  \href{https://arxiv.org/pdf/2308.06996}{arXiv:2308.06996 [math.DG]},  (2023).

\bibitem{Rong1996}
{\sc X.~Rong}, {\em On the fundamental groups of manifolds of positive
  sectional curvature}, Ann. of Math. (2), 143 (1996), pp.~397--411.

\bibitem{Rong2002}
\leavevmode\vrule height 2pt depth -1.6pt width 23pt, {\em Positively curved
  manifolds with almost maximal symmetry rank}, in Proceedings of the
  {C}onference on {G}eometric and {C}ombinatorial {G}roup {T}heory, {P}art {II}
  ({H}aifa, 2000), vol.~95, 2002, pp.~157--182.

\bibitem{Searle2023}
{\sc C.~Searle}, {\em Symmetries of spaces with lower curvature bounds},
  Notices Amer. Math. Soc., 70 (2023), pp.~564--575.

\bibitem{SearleSolorzanoWilhelm2015}
{\sc C.~Searle, P.~Sol\'{o}rzano, and F.~Wilhelm}, {\em Regularization via
  {C}heeger deformations}, Ann. Global Anal. Geom., 48 (2015), pp.~295--303.

\bibitem{Stefan1974}
{\sc P.~Stefan}, {\em Accessible sets, orbits, and foliations with
  singularities}, Proc. London Math. Soc. (3), 29 (1974), pp.~699--713.

\bibitem{Sussmann73}
{\sc H.~J. Sussmann}, {\em Orbits of families of vector fields and
  integrability of distributions}, Trans. Amer. Math. Soc., 180 (1973),
  pp.~171--188.

\bibitem{FloresTorres2024}
{\sc J.~F. Torres}, {\em Moduli space of bi-invariant metrics}, Bol. Soc. Mat.
  Mex. (3), 30 (2024), pp.~Paper No. 49, 17.

\bibitem{Wilking2003}
{\sc B.~Wilking}, {\em Torus actions on manifolds of positive sectional
  curvature}, Acta Math., 191 (2003), pp.~259--297.

\bibitem{Yamabe1950}
{\sc H.~Yamabe}, {\em On an arcwise connected subgroup of a {L}ie group}, Osaka
  Math. J., 2 (1950), pp.~13--14.

\end{thebibliography}
\end{document}